\documentclass[preprint,11pt]{elsarticle}

\usepackage{amssymb}
\usepackage{amsthm}
\usepackage[section]{placeins}
\usepackage{tikz}
\usepackage{tikz,pgfplots}
\usetikzlibrary{decorations.markings}

\usepackage[margin=2.5cm]{geometry}% by courtesy of Mico

\journal{Graphs and Combinatorics}

\newcommand{\Z}{\mathbb{Z}}
\newcommand{\Cay}{\mathrm{Cay}}
\newcommand{\divides}{\mathrel{|}}
\newcommand{\Kautz}{\mathrm{Kautz}}

\begin{document}

	\newdefinition{definition}{Definition}
	\newtheorem{theorem}{Theorem}
	\newtheorem{corollary}{Corollary}
	\newtheorem{lemma}{Lemma}
	\newtheorem{conjecture}{Conjecture}
	
	\tikzset{middlearrow/.style={
			decoration={markings,
				mark= at position 0.7 with {\arrow[scale=2]{#1}} ,
			},
			postaction={decorate}
		}
	}
	
	\tikzset{midarrow/.style={
			decoration={markings,
				mark= at position 0.5 with {\arrow[scale=2]{#1}} ,
			},
			postaction={decorate}
		}
	}
	
	\newcommand{\TODO}[1]{\textcolor{red}{TODO: #1}}

	\begin{frontmatter}
		
		
		\title{On networks with order close to the Moore bound}
		


		\author[label1]{James Tuite}
		\ead{james.tuite@open.ac.uk}
		\author[label2]{Grahame Erskine}

		\address{Department of Mathematics and Statistics, Open University, Walton Hall, Milton Keynes}

		\begin{abstract}
			The degree/diameter problem for mixed graphs asks for the largest possible order of a mixed graph with given diameter and degree parameters. Similarly the \emph{degree/geodecity} problem concerns the smallest order of a $k$-geodetic mixed graph with given minimum undirected and directed degrees; this is a generalisation of the classical degree/girth problem. In this paper we present new bounds on the order of mixed graphs with given diameter or geodetic girth and exhibit new examples of directed and mixed geodetic cages. In particular, we show that any $k$-geodetic mixed graph with excess one must have geodetic girth two and be totally regular, thereby proving an earlier conjecture of the authors.
		\end{abstract}
		
		
	\end{frontmatter}
	
	\newcommand*{\Perm}[2]{{}^{#1}\!P_{#2}}	
	
	\section{Introduction}
	
	It is often of practical interest to consider networks that include both undirected edges and directed arcs. For example, road networks contain both two-way and one-way streets and websites contain links that are unidirectional and others that are bidirectional. Such networks are represented mathematically by \emph{mixed graphs}; such graphs have applications in job scheduling~\cite{Ries} and Bayesian inference~\cite{CowDawLauSpi} amongst others. The efficiency of such networks may be measured by studying such graph parameters as the \emph{diameter} (longest distance between nodes) or the \emph{geodetic girth} (which pertains to the existence of multiple short paths between nodes). In this paper we discuss two extremal problems for these parameters in mixed graphs.
	
	In the case of undirected graphs, the \emph{degree/diameter} problem asks for the largest possible order of a graph with given diameter and maximum degree.  The order of such a graph is bounded above by the so-called \emph{Moore bound}; a survey of this problem can be found in~\cite{MilSir}. The \emph{degree/girth} problem requires the smallest possible order of a graph with given minimum degree and girth; a good survey of this problem is~\cite{ExoJaj}.  For this problem the Moore bound now serves as a lower bound on the order. The degree/diameter problem has also been investigated in the setting of directed graphs~\cite{MilSir} and mixed graphs~\cite{LopPer}. Several recent papers, such as~\cite{MilMirSil,Sil}, have also discussed a directed analogue of the degree/girth problem called the \emph{degree/geodecity} problem. In \cite{TuiErs} the present authors extended the degree/geodecity problem to mixed graphs and discussed the total regularity of extremal graphs in the degree/diameter and degree/geodecity problems.    
	
	The structure of this paper is as follows. Section~\ref{Notation} defines the notation that we will be using and provides some background on the problems that we will discuss. In Section~\ref{Existence of cages} we prove the existence of mixed geodetic cages and discuss monotonicity relations.  We then present strong new bounds on the excess of totally regular mixed graphs in Section~\ref{excesstotregular} and generalise our results to mixed graphs that are not totally regular in Section~\ref{counting not totreg}, which allows us to prove the non-existence of $k$-geodetic mixed graphs with excess one for $k \geq 3$, thereby proving a conjecture of the authors in~\cite{TuiErs}. Employing similar counting arguments, we give a new upper bound on the order of totally regular mixed graphs with undirected degree and directed degree equal to one in Section~\ref{improved Fiol bound}.  Finally in Section~\ref{directed and mixed cages} we present new mixed and directed geodetic cages and give upper bounds for some other values of the degrees and geodetic girth using a computer search among mixed Cayley graphs.

	\section{Notation}\label{Notation}
	
	Formally, a \emph{mixed graph} $G$ consists of a set $V(G)$ of \emph{vertices}, a set $E(G)$ of undirected \emph{edges} and a set $A(G)$ of directed \emph{arcs}.  An undirected edge is an unordered pair of vertices, whereas an arc is an ordered pair of vertices.  We forbid loops as well as parallel edges and arcs.  For any notation not defined here we refer to~\cite{BonMur}.
	
	Each vertex $u$ is incident with a certain number $d(u)$ of undirected edges; we call this the \emph{undirected degree} of $u$.  Similarly the number of arcs with initial point $u$ is the \emph{directed out-degree} of $u$ and is denoted $d^+(u)$, whereas the \emph{directed in-degree} of $u$ is the number of arcs of $G$ with terminal vertex $u$ and is written $d^-(u)$.  If there is an edge between vertices $u$ and $v$ we write $u \sim v$, whereas the presence of an arc from $u$ to $v$ is indicated by $u \rightarrow v$.  For any vertex $u$ we set $U(u) = \{ u_1,u_2, \dots , u_r\}  = \{ v \in V(G) : u \sim v\} $, $Z^-(u) = \{ v_1,v_2, \dots , v_s\} = \{ v \in V(G) : v \rightarrow u\} $ and $Z^+(u) = \{ u_{r+1}, \dots , u_{r+z}\} = \{ v \in V(G) : u \rightarrow v\} $.  If there exist $r$ and $z$ such that for all vertices $u$ we have $d(u) = r, d^+(u) = z$, then $G$ is said to be \emph{out-regular}.  If we also have $d^-(u) = d^+(u) = z$ for all $u$ then we say that $G$ is \emph{totally regular}.  If $G^U$ and $G^Z$ denote respectively the undirected and directed subgraphs (i.e. the subgraphs induced by the edges/arcs), then it can be seen that $G$ is out-regular if and only if $G^U$ is regular and $G^Z$ is out-regular, and $G$ is totally regular if and only if $G^U$ is regular and $G^Z$ is diregular.
	
	A \emph{walk} $W$ in $G$ is a sequence $u_0u_1 \dots u_{\ell }$ of vertices of $G$ such that for $0 \leq i \leq \ell -1$ either $u_i \sim u_{i+1}$ or $u_i \rightarrow u_{i+1}$.  The length of the walk is $\ell $ and $u_0$ and $u_{\ell }$ are the initial and terminal vertices of $W$ respectively.  The walk is \emph{non-backtracking} if the walk does not cross an edge and then immediately retrace it, i.e. if the walk does not contain a subsequence $u_i \sim u_{i+1} \sim u_i$.  We will call a non-backtracking walk in $G$ a \emph{mixed path}.
	
	The distance $d(u,v)$ from a vertex $u$ to a vertex $v$ is the length of a shortest mixed path with initial vertex $u$ and terminal vertex $v$.  Observe that we can have $d(u,v) \neq d(v,u)$.  If there is no mixed path from $u$ to $v$ then we set $d(u,v) = \infty $.  The diameter of $G$ is defined to be $diam(G) = \max \{ d(u,v): u,v \in V(G)\} $.  Suppose that for any ordered pair of vertices $(u,v)$ of $G$ there is at most one mixed path from $u$ to $v$ with length $\leq k$; then we say that $G$ is \emph{$k$-geodetic}.  The largest $k$ such that $G$ is $k$-geodetic is the \emph{geodetic girth} or \emph{geodecity} of $G$.    
	
	A \emph{mixed Moore graph} is an out-regular mixed graph $G$ such that for every pair of vertices $u,v$ of $G$ there is a unique mixed path of length $\leq k$ from $u$ to $v$. We can draw a \emph{mixed Moore tree} to deduce an upper bound on the order of a mixed graph $G$ with maximum undirected degree $r$, maximum directed out-degree $z$ and diameter $k$. Fix a vertex $u$ and call this root vertex Level $0$ of the tree. Draw edges from Level $0$ to Level 1 from $u$ to all of the undirected neighbours of $u$ and arcs from Level $0$ to all of the directed out-neighbours of $u$. In general, once we have added all vertices at Level $t$, where $0 \leq t \leq k-1$, we add the next level to the tree by the following rule for each vertex $u_i$ in Level $t$:
	
	\begin{itemize}
		\item Draw arcs from Level $t$ to Level $t+1$ from $u_i$ to all directed out-neighbours of $u_i$.
		\item If $u_i$ appears in Level $t$ as the terminal vertex of an arc from Level $t-1$ then draw edges from Level $t$ to Level $t+1$ from $u_i$ to all undirected neighbours of $u_i$.
		\item If $u_i$ appears in Level $t$ as the endpoint of an edge from a vertex $u_j$ in Level $t-1$, then below $u_i$ in the Moore tree draw an edge from $u_i$ to Level $t+1$ to all undirected neighbours of $u_i$ apart from $u_j$.    
	\end{itemize}
	
	We continue this process until we have a tree of depth $k$. As $G$ has diameter $k$ all vertices of $G$ are contained in the mixed Moore tree. An example for a mixed graph with maximum undirected degree $r = 3$, maximum directed out-degree $z = 3$ and diameter $k = 2$ is shown in Figure \ref{fig:mooretree}.

	\begin{figure}
		\centering
		\begin{tikzpicture}[middlearrow=stealth,x=0.2mm,y=-0.2mm,inner sep=0.1mm,scale=2.1,
			thick,vertex/.style={circle,draw,minimum size=10,font=\small,fill=lightgray},every label/.style={font=\scriptsize}]
			
			\node at (200,0) [vertex,label=above:{$0$}] (v0) {};
			
			\node at (25,100) [vertex,label=left:{$1$}] (v1) {};
			\node at (95,100) [vertex,label=left:{$2$}] (v2) {};
			\node at (165,100) [vertex,label=left:{$3$}] (v3) {};
			\node at (235,100) [vertex,label=left:{$4$}] (v4) {};
			\node at (305,100) [vertex,label=left:{$5$}] (v5) {};
			\node at (375,100) [vertex,label=left:{$6$}] (v6) {};
			
			\node at (5,200) [vertex,label=below:{$7$}] (v11) {};
			\node at (15,200) [vertex,label=below:{$8$}] (v12) {};
			\node at (25,200) [vertex,label=below:{$9$}] (v13) {};
			\node at (35,200) [vertex,label=below:{$10$}] (v14) {};
			\node at (45,200) [vertex,label=below:{$11$}] (v15) {};
			
			\node at (75,200) [vertex,label=below:{$12$}] (v21) {};
			\node at (85,200) [vertex,label=below:{$13$}] (v22) {};
			\node at (95,200) [vertex,label=below:{$14$}] (v23) {};
			\node at (105,200) [vertex,label=below:{$15$}] (v24) {};
			\node at (115,200) [vertex,label=below:{$16$}] (v25) {};
			
			\node at (145,200) [vertex,label=below:{$17$}] (v31) {};
			\node at (155,200) [vertex,label=below:{$18$}] (v32) {};
			\node at (165,200) [vertex,label=below:{$19$}] (v33) {};
			\node at (175,200) [vertex,label=below:{$20$}] (v34) {};
			\node at (185,200) [vertex,label=below:{$21$}] (v35) {};
			
			\node at (210,200) [vertex,label=below:{$22$}] (v41) {};
			\node at (220,200) [vertex,label=below:{$23$}] (v42) {};
			\node at (230,200) [vertex,label=below:{$24$}] (v43) {};
			\node at (240,200) [vertex,label=below:{$25$}] (v44) {};
			\node at (250,200) [vertex,label=below:{$26$}] (v45) {};
			\node at (260,200) [vertex,label=below:{$27$}] (v46) {};
			
			\node at (280,200) [vertex,label=below:{$28$}] (v51) {};
			\node at (290,200) [vertex,label=below:{$29$}] (v52) {};
			\node at (300,200) [vertex,label=below:{$30$}] (v53) {};
			\node at (310,200) [vertex,label=below:{$31$}] (v54) {};
			\node at (320,200) [vertex,label=below:{$32$}] (v55) {};
			\node at (330,200) [vertex,label=below:{$33$}] (v56) {};
			
			\node at (350,200) [vertex,label=below:{$34$}g] (v61) {};
			\node at (360,200) [vertex,label=below:{$35$}] (v62) {};
			\node at (370,200) [vertex,label=below:{$36$}] (v63) {};
			\node at (380,200) [vertex,label=below:{$37$}] (v64) {};
			\node at (390,200) [vertex,label=below:{$38$}] (v65) {};
			\node at (400,200) [vertex,label=below:{$39$}] (v66) {};
			
			\path
			(v0) edge (v1)
			(v0) edge (v2)
			(v0) edge (v3)
			(v0) edge [middlearrow] (v4)
			(v0) edge [middlearrow] (v5)
			(v0) edge [middlearrow] (v6)
			
			(v1) edge (v11)
			(v1) edge (v12)
			(v1) edge [middlearrow] (v13)
			(v1) edge [middlearrow] (v14)
			(v1) edge [middlearrow] (v15)
			
			(v2) edge (v21)
			(v2) edge (v22)
			(v2) edge [middlearrow] (v23)
			(v2) edge [middlearrow] (v24)
			(v2) edge [middlearrow] (v25)
			
			(v3) edge (v31)
			(v3) edge (v32)
			(v3) edge [middlearrow] (v33)
			(v3) edge [middlearrow] (v34)
			(v3) edge [middlearrow] (v35)
			
			(v4) edge (v41)
			(v4) edge (v42)
			(v4) edge (v43)
			(v4) edge [middlearrow] (v44)
			(v4) edge [middlearrow] (v45)
			(v4) edge [middlearrow] (v46)
			
			(v5) edge (v51)
			(v5) edge (v52)
			(v5) edge (v53)
			(v5) edge [middlearrow] (v54)
			(v5) edge [middlearrow] (v55)
			(v5) edge [middlearrow] (v56)
			
			(v6) edge (v61)
			(v6) edge (v62)
			(v6) edge (v63)
			(v6) edge [middlearrow] (v64)
			(v6) edge [middlearrow] (v65)
			(v6) edge [middlearrow] (v66)
			;
		\end{tikzpicture}
		\caption{The Moore tree for $r = 3, z = 3, k = 2$}
		\label{fig:mooretree}
	\end{figure}
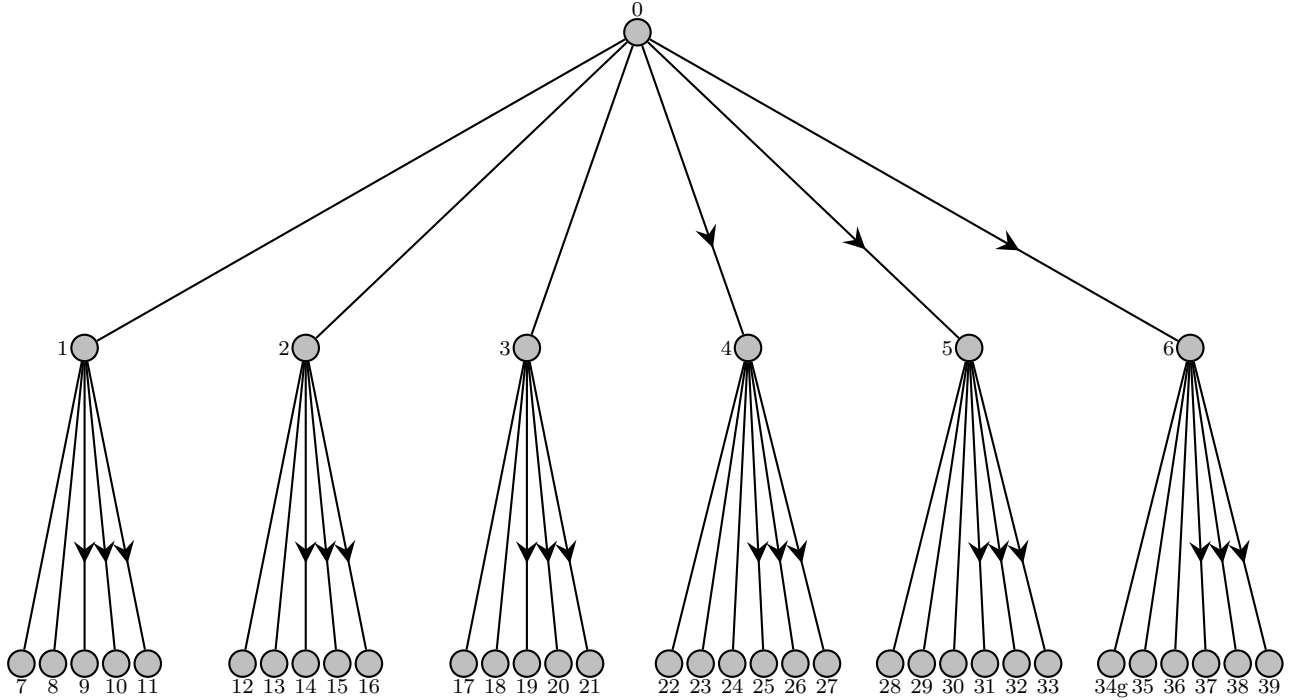
	
	Counting the number of vertices in the Moore tree therefore gives an upper bound (called the \emph{mixed Moore bound}) on the order of a mixed graph with given diameter. An exact expression for the Moore bound for mixed graphs was derived in \cite{BusAmiErsMilPer} using recurrence relations. 
	
	\begin{theorem}\cite{BusAmiErsMilPer}[Mixed Moore bound]
		The order of a mixed graph with maximum undirected degree $r$, maximum out-degree $z$ and diameter $k$ is bounded above by
		\[ M(r,z,k) = A\frac{u_1^{k+1}-1}{u_1-1}+B\frac{u_2^{k+1}-1}{u_2-1},\]
		where
		\[ v = (z+r)^2+2(z-r)+1, u_1 = \frac{z+r-1-\sqrt{v}}{2}, u_2 = \frac{z+r-1+\sqrt{v}}{2}\]
		and
		\[ A = \frac{\sqrt{v}-(z+r+1)}{2\sqrt{v}}, B = \frac{\sqrt{v}+(z+r+1)}{2\sqrt{v}}.\]  
	\end{theorem}
	If $r = 0$ or $z = 0$ then this expression reduces to the undirected and directed Moore bounds respectively.
	
	A graph that meets the mixed Moore bound is called a \emph{mixed Moore graph}. Recall that a mixed graph $G$ is $k$-geodetic if and only if for any pair $u,v$ of vertices of $G$ there is at most one mixed path (i.e. non-backtracking mixed walk) of length $\leq k$ from $u$ to $v$ in $G$. It is easily seen that a mixed graph is Moore if and only if it satisfies the following conditions.
	
	\begin{theorem}\label{Moore conditions} 
		A mixed graph $G$ is Moore if and only if
		\begin{itemize}
			\item $G$ is totally regular with undirected degree $r$ and directed degree $z$,
			\item the diameter of $G$ is $k$, and
			\item $G$ is $k$-geodetic.
		\end{itemize} 
	\end{theorem}
	
	Mixed Moore graphs with diameter $k = 2$ were first investigated by Bos\'{a}k in the seventies \cite{Bos2,Bos3,Bos}. In \cite{Bos} he proved that any mixed Moore graph is totally regular and used spectral methods to prove that the undirected degree $r$ and directed out-degree $z$ of a mixed Moore graph with diameter two satisfy a very special condition.
	
	\begin{theorem}\cite{Bos}\label{mixed Moore graph spectrum}\label{Bosak spectrum}
		Apart from trivial cases, if there exists a mixed Moore graph with diameter two, undirected degree $r$ and directed out-degree $z$, then there exists a positive odd integer $c$ such that	
		$c\divides(4z-3)(4z+5)$ and $c^2+3 = 4r$.
	\end{theorem}
	However, Theorem \ref{mixed Moore graph spectrum} leaves an infinite number of pairs $r,z$ for which the existence of a mixed Moore graph with undirected degree $r$, directed out-degree $z$ and diameter two is undecided. The smallest orders not covered by Theorem \ref{Bosak spectrum} are displayed in Table \ref{tab:values not covered by Bosak}.
	
	\begin{table}
		\begin{footnotesize}

			\begin{center}
				\begin{tabular}{|c|c|c| }
					\hline
					
					Undirected degree $r$ & Directed degree $z$ & Order $n$   \\
					\hline
					1 & any & $r^2+2r+3$ \\
					\hline
					3 & 1 & 18\\
					& 3 & 40\\
					& 4 & 54 \\
					& 6 & 88 \\
					& 7 & 108 \\
					& \dots & \dots \\
					\hline
					7 & 2 & 84 \\
					& 5 & 150 \\
					& 7 & 204 \\
					& \dots & \dots \\
					\hline
					13 & 4 & 294 \\
					& 6 & 368 \\
					& \dots & \dots \\
					\hline
					21 & 1 & 486 \\
					& \dots & \dots \\
					\hline
					\dots & \dots & \dots \\
					\hline
					
				\end{tabular}
			\end{center}
			
		\end{footnotesize}
		\caption{Values of $r$ and $z$ not covered by Theorem \ref{Bosak spectrum}}
		\label{tab:values not covered by Bosak}
	\end{table}
	
	There is one known infinite family of mixed Moore graphs with diameter two, formed by collapsing all digons in the Kautz digraph $K(d,k)$ into edges. This mixed graph can be described quite easily. Take an alphabet $\Omega $ of size $z+2$.  The vertices of $\Kautz (z)$ are words $ab$, where $a \not = b$.  For all $a,b,c \in \Omega $ with $a \not = b$ and $ b\not = c$ we introduce an arc $ab \rightarrow bc$ when $c \not = a$ and an edge $ab \sim ba$. It is easily verified that this yields a mixed Moore graph with undirected degree $r = 1$, directed out-degree $z$ and diameter $k = 2$. In fact it shown in \cite{Gim} using spectral techniques that these are the unique mixed Moore graphs with these parameters. 
	
	\begin{theorem}\cite{Gim}
		For all $z \geq 1$ there is a unique mixed Moore graph with undirected degree $r = 1$, directed out-degree $z$ and diameter $k = 2$. 	
	\end{theorem}
	In \cite{Bos} Bos\'{a}k identified a further mixed Moore graph with undirected degree $r = 3$, directed out-degree $z = 1$, diameter $k = 2$ and order $M(3,1,2) = 18$.  The uniqueness of this graph was proven in \cite{NguMilGim}. 
	
	One method of searching for mixed Moore graphs is to restrict the search space to Cayley mixed graphs. By carrying out a computer search for Cayley mixed graphs that meet the Moore bound J\o rgensen found two Cayley mixed Moore graphs with undirected degree $r = 3$, directed out-degree $z = 7$, diameter $k = 2$ and order $n = 108$ \cite{Jor}. However, it has been shown that there are no further Cayley mixed Moore graphs with diameter two and order $\leq 485$ \cite{Ers,LopPerPuj}. A search using a SAT solver has also completely ruled out the existence of mixed Moore graphs with diameter two and orders $40, 54$ or $84$ \cite{LopMirFer}. 
	
	It is natural to ask whether there exist any mixed Moore graphs with diameter greater than two? It was shown by a counting argument in \cite{NguMilGim} that the answer to this question is negative, except in trivial cases. 
	
	\begin{theorem}\cite{NguMilGim}\label{no mixed Moore graphs for k geq 3}
		There are no mixed Moore graphs with diameter $k \geq 3$, except for undirected and directed cycles.  	
	\end{theorem}  
	Whilst there remain an infinite number of open cases, it is evident that it is very difficult for a mixed graph to meet the mixed Moore bound. In general the mixed Moore tree of depth $k$ will either not contain all vertices of $G$ (in which case the diameter of $G$ is larger than $k$) or there will be vertices repeated in the Moore tree (in which case $G$ is not $k$-geodetic). It is therefore of interest to study the structure of mixed graphs with order close to the mixed Moore bound. To this end in the conditions in Theorem \ref{Moore conditions} we can either relax the requirement that all of the vertices in the Moore tree be distinct or the requirement that the Moore tree contains all of the vertices of $G$. This motivates the following definitions.
	
	\begin{definition}
		~
		\begin{itemize}
			\item A mixed graph with maximum undirected degree $r$, maximum directed out-degree $z$, diameter $k$ and order $M(r,z,k)-\delta $ is called an \emph{$(r,z,k;-\delta)$-graph} and has \emph{defect} $\delta $. A mixed graph with defect one is called an \emph{almost mixed Moore graph}.
			\item A $k$-geodetic mixed graph with minimum undirected degree $r$, minimum directed out-degree $z$ and order $M(r,z,k)+\epsilon $ is called an \emph{$(r,z,k;+\epsilon )$-graph} and has \emph{excess} $\epsilon $. The smallest possible value of $\epsilon $ such that there exists an $(r,z,k;+\epsilon )$-graph will be written $\epsilon (r,z,k)$. We set $N(r,z,k) = M(r,z,k)+\epsilon (r,z,k)$. 
		\end{itemize}	
	\end{definition}

	A graph with defect $\delta = 1$ is called an \emph{almost mixed Moore graph}.  It was shown by the present authors that any $(r,z,2;-1)$- or $(r,z,2;+1)$-graph must be totally regular \cite{TuiErs}.  L\'{o}pez and Miret used spectral theory to derive the following necessary condition for the existence of an almost mixed Moore graph with diameter $k = 2$ in \cite{LopMir}.
	
	\begin{theorem}
		Let $G$ be a (totally regular) $(r,z,2;-1)$-graph.  Then $r$ is even and one of the following three possibilities holds:
		\newline i) $r = 2$,
		\newline ii) there exists an odd integer $c$ such that $c^2 = 4r+1$ and $c\divides(4z+1)(4z-7)$, or
		\newline iii) there exists an odd integer $c$ such that $c^2 = 4r-7$ and $c\divides(16z^2+40z-23)$.
	\end{theorem}  
	Using the methods of L\'{o}pez and Miret it is possible to show the following result, which we state without proof.

	\begin{theorem}\label{spectral theorem}
		Let $G$ be a totally regular $(r,z,2,+1)$-graph.  Then either:
		\newline $r = 2$;
		\newline $4r+1 = c^2$ for some $c \in \mathbb{N}$ and $c\divides(16z^2-24z+25)$; or
		\newline $4r-7 = c^2$ for some $c \in \mathbb{N}$ and $c\divides(16z^2+40z+9)$.
	\end{theorem}
	Mixed graphs with defect and excess one are displayed in Figures \ref{fig:almost} and \ref{fig:excessone} respectively. The latter is the only known mixed graph with excess one.

	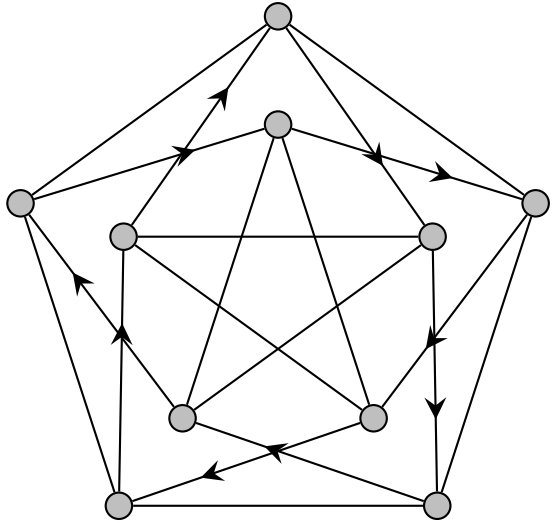
\begin{figure}
		\centering
		\begin{tikzpicture}[middlearrow=stealth,x=0.2mm,y=-0.2mm,inner sep=0.2mm,scale=1,thick,vertex/.style={circle,draw,minimum size=10,fill=lightgray}]
			\node at (380,200) [vertex] (v1) {};
			\node at (208.8,324.4) [vertex] (v2) {};
			\node at (274.2,525.6) [vertex] (v3) {};
			\node at (485.8,525.6) [vertex] (v4) {};
			\node at (551.2,324.4) [vertex] (v5) {};
			\node at (380,272) [vertex] (v6) {};
			\node at (316.5,467.4) [vertex] (v7) {};
			\node at (482.7,346.6) [vertex] (v8) {};
			\node at (277.3,346.6) [vertex] (v9) {};
			\node at (443.5,467.4) [vertex] (v10) {};
			\path
			(v1) edge (v2)
			(v1) edge (v5)
			(v2) edge (v3)
			(v3) edge (v4)
			(v4) edge (v5)
			(v6) edge (v7)
			(v6) edge (v10)
			(v7) edge (v8)
			(v8) edge (v9)
			(v9) edge (v10)
			(v2) edge [middlearrow] (v6)
			(v6) edge [middlearrow] (v5)
			(v5) edge [middlearrow] (v10)
			(v10) edge [middlearrow] (v3)
			(v3) edge [middlearrow] (v9)
			(v9) edge [middlearrow] (v1)
			(v1) edge [middlearrow] (v8)
			(v8) edge [middlearrow] (v4)
			(v4) edge [middlearrow] (v7)
			(v7) edge [middlearrow] (v2)
			;
		\end{tikzpicture}
		\caption{A mixed almost Moore graph}
		\label{fig:almost}
	\end{figure}

	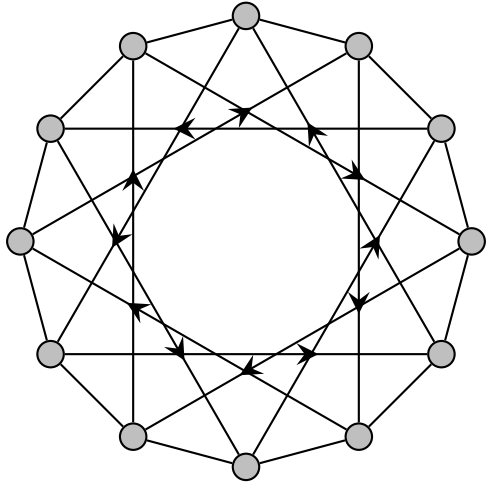
\begin{figure}[h]
		\centering
		\begin{tikzpicture}[middlearrow=stealth,x=0.2mm,y=-0.2mm,inner sep=0.2mm,scale=0.5,thick,vertex/.style={circle,draw,minimum size=10,fill=lightgray}]
			\node at (300,0) [vertex] (v0) {};
			\node at (259.8,150) [vertex] (v1) {};
			\node at (150,259.8) [vertex] (v2) {};
			\node at (0,300) [vertex] (v3) {};
			\node at (-150,259.8) [vertex] (v4) {};
			\node at (-259.8,150) [vertex] (v5) {};
			\node at (-300,0) [vertex] (v6) {};
			\node at (-259.8,-150) [vertex] (v7) {};
			\node at (-150,-259.8) [vertex] (v8) {};
			\node at (0,-300) [vertex] (v9) {};
			\node at (150,-259.8) [vertex] (v10) {};
			\node at (259.8,-150) [vertex] (v11) {};
			\path
			(v0) edge (v1)
			(v1) edge (v2)
			(v2) edge (v3)
			(v3) edge (v4)
			(v4) edge (v5)
			(v5) edge (v6)
			(v6) edge (v7)
			(v7) edge (v8)
			(v8) edge (v9)
			(v9) edge (v10)
			(v10) edge (v11)
			(v11) edge (v0)
			
			(v0) edge [middlearrow] (v4)
			(v4) edge [middlearrow] (v8)	
			(v8) edge [middlearrow] (v0)
			(v1) edge [middlearrow] (v9)
			(v9) edge [middlearrow] (v5)	
			(v5) edge [middlearrow] (v1)
			(v2) edge [middlearrow] (v6)
			(v6) edge [middlearrow] (v10)	
			(v10) edge [middlearrow] (v2)
			(v3) edge [middlearrow] (v11)
			(v11) edge [middlearrow] (v7)	
			(v7) edge [middlearrow] (v3)
			
			;
		\end{tikzpicture}
		\caption{A mixed graph with excess one}
		\label{fig:excessone}
	\end{figure}
	If $G$ is an $(r,z,k;-\delta )$-graph, then there will be exactly $\delta $ repetitions in the Moore tree of depth $k$ based at any vertex $u$.  We form a multiset $R(u)$, with a vertex $v$ appearing $t-1$ times in $R(u)$ if it appears $t$ times in the Moore tree; $R(u)$ is called the \emph{repeat set} of $u$.  Similarly if $G$ is an $(r,z,k;+\epsilon )$-graph then all vertices appearing in the Moore tree based at a vertex $u$ will be distinct, but the tree will not contain all vertices of $G$.  If $G$ is totally regular, then there will be exactly $\epsilon $ vertices $v$ satisfying $d(u,v) \geq k+1$; any such $v$ is an \emph{outlier} of $u$ and the set $O(u)$ of all outliers of $u$ is the \emph{outlier set} of $u$ (observe that there are no repetitions in $O(u)$).  For given $r,z$ and $k$ a mixed graph with smallest possible excess is an \emph{$(r,z,k)$-geodetic-cage}, or a \emph{geodetic cage} if the values of the parameters can be inferred from the context (we insert `geodetic' to distinguish such constructions from extremal cycle-avoiding mixed graphs, which are already called cages \cite{AraHerMon}).

	\section{Existence of cages}\label{Existence of cages}
	
	One subtle point that does not arise in the degree/diameter problem is that it is not immediately clear that cages exist for all values of the degree $d$ and girth $g$; therefore it is necessary to prove that for any $d \geq 2$ and $g \geq 3$ there exists a graph with degree $d$ and girth $g$.  This result was first shown by Sachs in \cite{Sac} using a recursive construction. The upper bound in \cite{Sac} was subsequently improved in a joint paper by Sachs and Erd\H{o}s \cite{ErdSac}.  An approachable presentation of these proofs is given in an appendix of \cite{ExoJaj}. 
	
	Our first step is therefore to show that geodetic cages exist for all values of the geodetic girth $k$ and the degree parameters $r$ and $z$. In the purely directed case we obtain the existence of geodetic cages and a good estimate of their order almost for free from a nice family of digraphs called the \emph{permutation digraphs}. These digraphs were first mentioned in \cite{Fio} and their properties further developed in \cite{BruFioFio}. These digraphs are defined as follows.

	\begin{definition}
		For $d,k \geq 2$ the vertex set of the permutation digraph $P(d,k)$ consists of all sequences $x_0x_1\dots x_{k-1}$ of length $k$ drawn from an alphabet $[d+k] = \{ 0,1,2,\dots , d+k-1\} $ such that for $0 \leq i < j \leq k-1$ we have $x_i \not = x_j$. 
		
		The adjacencies of $P(d,k)$ are defined by \[ x_0x_1\dots x_{k-1} \rightarrow x_1x_2\dots x_{k-1}x_k,\] where $x_k \in ([d+k] - \{ x_0,x_1,\dots , x_{k-1}\} )$.    
	\end{definition}
	
	It is shown in \cite{BruFioFio} that permutation digraphs are highly symmetric. The symmetric group on $d+k$ symbols acts on $P(d,k)$ in a natural way by permuting the symbols of the underlying alphabet, meaning that they are arc-transitive, although not 2-arc-transitive. The symmetry groups of the permutation digraphs are derived and the Cayley permutation digraphs classified in \cite{BruFioFio}. The important property of the permutation digraphs from our point of view is that $P(d,k)$ is $k$-geodetic and for fixed $k \geq 2$ the digraphs $P(d,k)$ have order approaching the directed Moore bound $M(d,k)$ asymptotically from above. The digraph $P(2,2)$ is displayed in Figure \ref{fig:P22}.
	
	\begin{lemma}\label{permutation order}
		For $d,k \geq 2$ the permutation digraph $P(d,k)$ is diregular with degree $d$, has geodetic girth $k$ and has order \[ \Perm{d+k}{k} = (d+k)(d+k-1)\dots(d+1).\] Hence for fixed $k \geq 2$ the excess of $P(d,k)$ is \[ ((d+k)(d+k-1)\dots (d+1)) - (d^k+d^{k-1}+\dots +d+1) \sim \left(\frac{k(k+1)}{2}-1\right )d^{k-1} \] as $d \rightarrow \infty $.
	\end{lemma}
	\begin{proof}
		For all $d,k \geq 2$ the digraph $P(d,k)$ contains directed cycles of length $k+1$, for example
		\[ 012\dots (k-1) \rightarrow 12\dots (k-1)k \rightarrow 23\dots(k-1)k0 \rightarrow \dots \rightarrow k01\dots (k-2) \rightarrow 012\dots(k-1),\] so the geodetic girth of $P(d,k)$ is certainly $\leq k$. 
		
		By vertex-transitivity of $P(d,k)$, to prove $k$-geodecity it is sufficient to demonstrate that if $P$ and $Q$ are $\leq k$-paths in $P(d,k)$ from the vertex $012\dots (k-1)$ to a vertex $x_0x_1\dots x_{k-1}$, then $P = Q$. All vertices at distance $r \leq k-1$ from $012\dots (k-1)$ have first symbol $r$, whereas all vertices at distance $k$ from $012\dots (k-1)$ have a first symbol that does not lie in $\{ 0,1,\dots ,k-1\} $.  As $d(01\dots (k-1),x_0x_1\dots x_{k-1}) \leq k$ by assumption, it follows that if $x_0 \in \{ 0,1,\dots ,k-1\} $ then both $P$ and $Q$ have length $x_0$, whereas if $x_0 \not \in \{ 0,1,\dots ,k-1\} $ then both $P$ and $Q$ must have length $k$; in either case $l(P) = l(Q)$. 
		
		If $x_0 = r \in \{ 0,1,\dots ,k-1\} $ then the only path with length $r$ from $01\dots (k-1)$ to $x_0x_1\dots x_{k-1}$ is the path with initial vertex $01\dots (k-1)$ obtained by successively deleting the symbol $i$ on the left-hand side and adding the symbol $x_{k-r+i}$ on the right for $i = 0,1,\dots ,r-1$. If $x_0 \not \in \{ 0,1,\dots ,k-1\} $, then the first arc $e$ of both $P$ and $Q$ must be $012\dots (k-1)\rightarrow 12 \dots (k-1)x_0$.  Deleting the arc $e$ from $P$ and $Q$ leaves two paths $P'$ and $Q'$ of length $k-1$ from $12\dots (k-1)x_0$ to $x_0x_1\dots x_{k-1}$; by the above reasoning $P' = Q'$ and hence $P = Q$.  
	\end{proof}

	\begin{figure}\centering
		\begin{tikzpicture}[midarrow=stealth,x=0.2mm,y=-0.2mm,inner sep=0.5mm,scale=2,
			thick,vertex/.style={circle,draw,minimum size=10,font=\footnotesize,fill=white},edge label/.style={fill=white}]
			\tiny
			\node at (0,100) [vertex] (v0) {10};
			\node at (-86.6025,50) [vertex] (v1) {31};
			\node at (-86.6025,-50) [vertex] (v2) {12};
			\node at (0,-100) [vertex] (v3) {01};
			\node at (86.6025,-50) [vertex] (v4) {13};
			\node at (86.6025,50) [vertex] (v5) {21};
			\node at (-35,60.6218) [vertex] (v6) {03};
			\node at (-70,0) [vertex] (v7) {23};
			\node at (-35,-60.6218) [vertex] (v8) {20};
			\node at (35,-60.6218) [vertex] (v9) {30};
			\node at (70,0) [vertex] (v10) {32};
			\node at (35,60.6218) [vertex] (v11) {02};
			
			\path
			(v0) edge [midarrow] (v6)
			(v0) edge [midarrow] (v11)
			(v1) edge [midarrow] (v0)
			(v1) edge [midarrow] (v2)
			(v2) edge [midarrow] (v7)
			(v2) edge [midarrow] (v8)
			(v3) edge [midarrow] (v2)
			(v3) edge [midarrow] (v4)
			(v4) edge [midarrow] (v9)
			(v4) edge [midarrow] (v10)
			(v5) edge [midarrow] (v4)
			(v5) edge [midarrow] (v0)
			(v6) edge [midarrow] (v1)
			(v6) edge [midarrow] (v10)
			(v7) edge [midarrow] (v1)
			(v7) edge [midarrow] (v9)
			(v8) edge [midarrow] (v3)
			(v8) edge [midarrow] (v6)
			(v9) edge [midarrow] (v3)
			(v9) edge [midarrow] (v11)
			(v10) edge [midarrow] (v5)
			(v10) edge [midarrow] (v8)
			(v11) edge [midarrow] (v5)
			(v11) edge [midarrow] (v7)
			
			;

		\end{tikzpicture}
		\caption{$P(2,2)$}
		\label{fig:P22}
	\end{figure}
	Lemma \ref{permutation order} shows that a $(0,z,k)$-geodetic cage exists for all values of $z,k \geq 1$. As the permutation digraphs are diregular, we see that for $d,k \geq 1$ there is also a smallest possible diregular $k$-geodetic digraph with degree $d$. By combining this construction with that of Sachs \cite{Sac} for undirected graphs we can show the existence of mixed geodetic cages for all $r,z,k \geq 1$.

	\begin{theorem}\label{mixed cages exist}
		There exists a mixed geodetic $(r,z,k)$-cage for all $r, z, k \geq 1$.
	\end{theorem}
	\begin{proof}
		We employ a truncation argument.  Let $H$ be an undirected cage with degree $r$, girth $g = 2k+1$ and order $n$.  Let $H'$ be a directed geodetic cage with geodetic girth $k$ and directed out-degree $nz$.  We form a mixed graph $G$ by identifying each vertex $u$ of $H'$ with an isomorphic copy $H_u$ of $H$ and connecting the copies of $H$ by arcs in accordance with the topology of $H'$; specifically, for each vertex $u$ of $H'$ partition the $nz$ arcs from $u$ in $H'$ into $n$ sets $A_1,A_2,\dots ,A_n$ of $z$ arcs and assign a set $A_i$ of arcs to each of the $n$ vertices in $H_u$, such that if an arc in $A_i$ goes to a vertex $v$ in $H'$, then in $G$ it is directed to any vertex of $H_v$. The resulting mixed graph $G$ obviously has geodetic girth $k$.  A similar construction starting with directed cages substituted for vertices of an undirected cage establishes the other part of the theorem.
		
	\end{proof}
	
	As in the undirected degree/girth problem, the bounds given in Theorem \ref{mixed cages exist} are much too large to be of any practical help. We also note that by using regular graphs with girth $2k+1$ (which exist by \cite{Sac}) and diregular $k$-geodetic digraphs (we can use the permutation digraphs), the truncation argument in Theorem \ref{mixed cages exist} also shows the existence of a smallest totally regular $(r,z,k;+\epsilon )$-graph.
	
	\begin{corollary}\label{smallest totally regular mixed graph}
		For all $r,z \geq 1$ and $k \geq 2$ there exists a smallest totally regular $(r,z,k;+\epsilon )$-graph.
	\end{corollary} 
	
	Now that the existence of mixed geodetic cages has been established, the question of monotonicity arises.  Intuition suggests that the order of a cage should grow strictly with increasing $r, z$ and $k$.  Monotonicity of the order of cages in the undirected degree/girth problem was proven by Fu, Huang and Rodger \cite{FuHuaRod} and degree monotonicity of undirected cages was discussed in \cite{WanYu}, but appears to be a difficult problem.  We generalise the approach of \cite{FuHuaRod} to prove strict monotonicity of the order of mixed cages in the geodetic girth $k$. The following proof also applies to purely directed geodetic cages.  
	
	\begin{theorem}\label{mixed monotonicity k}
		$N(r,z,k) < N(r,z,k+1)$ for all $k \geq 2$.
	\end{theorem}
	\begin{proof}
		Let $G$ be an $(r,z,k+1)$-cage. Suppose that there exists a vertex $u$ of $G$ with even undirected degree $d(u)$.  Write $U(u) = \{ u_1,u_2, \dots , u_{2r-1},u_{2r}\} $.  Define the graph $G'$ as follows: delete $u$ from $G$, join $u_{2i-1}$ to $u_{2i}$ by an undirected edge for $1 \leq i \leq r$ and for every vertex $u^-$ in $Z^-(u)$ insert an arc from $u^-$ to some vertex $u^+$ in $Z^+(u)$. This construction is shown in Figure~\ref{Construction for mono cages}, with the new arcs and edges shown dashed. Call the added arcs and edges \emph{new elements}.
		
		Suppose that $G'$ is not $k$-geodetic; let $w$ and $w'$ be vertices of $G'$ with distinct mixed paths $P$ and $Q$ of length $\leq k$ from $w$ to $w'$. As each new element in $G'$ can be extended to a walk of length two in $G$ whilst preserving the non-backtracking property and $G$ is $(k+1)$-geodetic, we can assume that the mixed path $P$ contains at least two new elements. Examining a mixed path with length $\leq k-2$ between adjacent new elements in $P$, we see that there exists a non-backtracking closed walk of length $k$ through $u$ in $G$, which is impossible. Thus $G$ is at least $k$-geodetic and, having order smaller than the $(r,z,k+1)$-cage $G$, its geodetic girth must be exactly $k$.
		
		Thus we can assume that every vertex of $G$ has odd undirected degree. Let $u \sim v$ be an undirected edge of $G$. Let $U(u) = \{ v,u_1,u_2,\dots , u_{2r}\} $ and $U(v) = \{ u,v_1,v_2,\dots, v_{2s}\} $.  Form a new graph $G''$ by deleting $u$ and $v$ and matching up the remaining neighbours of $u$ and $v$ by new elements as in the previous construction, i.e. setting $u_{2i-1} \sim u_{2i}$ for $1 \leq i \leq r$, $v_{2j-1} \sim v_{2j}$ for $1 \leq j \leq s$ and inserting an arc from each vertex of $Z^-(u)$ to $Z^+(u)$ and an arc from each vertex of $Z^-(v)$ to $Z^+(v)$. Assuming $k \geq 2$, notice that the sets $U(u)-\{ v\} $, $U(v)-\{ u\} $, $Z^-(u)$, $Z^+(u)$, $Z^-(v)$ and $Z^+(v)$ are pairwise disjoint.   
		
		If $G''$ has geodetic girth $\leq k-1$, with two distinct mixed paths $P$ and $Q$ from a vertex $w$ to a vertex $w'$, then as before we can assume that $P$ contains two new elements. 
		
		Consider consecutive new elements in $P$. By the preceding argument these new elements cannot be associated with same vertex in $G$, for example a new edge between undirected neighbours of $u$ and an arc from $Z^-(u)$ to $Z^+(u)$ would yield a contradiction as above. By symmetry we can assume that the first element is associated with $u$ and the second with $v$; for example, these elements could be a new arc from $Z^-(u)$ to $Z^+(u)$ followed by a new edge in $U(v)$. Looking at the mixed subpath of $P$ between these consecutive new elements, we see that in $G$ there is either a mixed path of length $\leq k-2$ from $N^+(u)$ to $N^-(v)$; it follows that there are distinct mixed paths of length $\leq k$ from $u$ to $v$ in $G$, a contradiction, so $G''$ is $k$-geodetic.  
	\end{proof}
	
	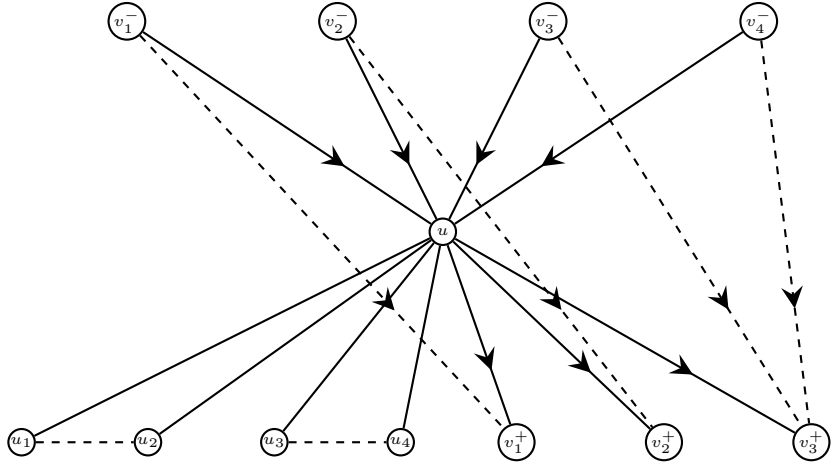
\begin{figure}\centering
		\begin{tikzpicture}[middlearrow=stealth,x=0.2mm,y=-0.2mm,inner sep=0.1mm,scale=1.4,
			thick,vertex/.style={circle,draw,minimum size=10,font=\tiny,fill=white},edge label/.style={fill=white}]
			\tiny
			\node at (200,0) [vertex] (v0) {$u$};
			\node at (0,100) [vertex] (v1) {$u_1$};
			\node at (60,100) [vertex] (v2) {$u_2$};
			\node at (120,100) [vertex] (v3) {$u_3$};
			\node at (180,100) [vertex] (v15) {$u_4$};
			\node at (235,100) [vertex] (v4) {$v_1^+$};
			\node at (305,100) [vertex] (v5) {$v_2^+$};
			\node at (375,100) [vertex] (v6) {$v_3^+$};
			\node at (50,-100) [vertex] (v11) {$v_1^-$};
			\node at (150,-100) [vertex] (v12) {$v_2^-$};
			\node at (250,-100) [vertex] (v13) {$v_3^-$};
			\node at (350,-100) [vertex] (v14) {$v_4^-$};

			\path
			(v0) edge (v1)
			(v0) edge (v2)
			(v0) edge (v3)
			(v0) edge (v15)
			(v0) edge  [middlearrow] (v4)
			(v0) edge  [middlearrow] (v5)
			(v0) edge  [middlearrow] (v6)
			(v11) edge  [middlearrow] (v0)
			(v12) edge  [middlearrow] (v0)
			(v13) edge  [middlearrow] (v0)
			(v14) edge  [middlearrow] (v0)	
			(v1) edge [dashed] (v2)
			(v3) edge [dashed] (v15)
			(v11) edge  [middlearrow, dashed] (v4)
			(v12) edge  [middlearrow, dashed] (v5)
			(v13) edge  [middlearrow, dashed] (v6)
			(v14) edge  [middlearrow, dashed] (v6)	
			;
		\end{tikzpicture}
		\caption{The construction in Theorem \ref{mixed monotonicity k}}
		\label{Construction for mono cages}
	\end{figure}

	Applying the procedure of Theorem \ref{mixed monotonicity k} to a smallest totally regular $(r,z,k;+\epsilon )$-graph (which we know to exist by Corollary \ref{smallest totally regular mixed graph}) by joining vertices of $Z^-(u)$ to $Z^+(u)$ by arcs in a one-to-one fashion, we see that strict monotonicity in the geodetic girth $k$ also holds for the order of smallest possible totally regular $(r,z,k;+\epsilon )$-graphs. Monotonicity in the directed out-degree is simple to demonstrate.
	
	\begin{theorem}
		$N(r,z,k) \leq N(r,z+1,k)$.  If $r = 0$, then strict inequality holds.
	\end{theorem}
	\begin{proof}
		Let $G$ be an $(r,z+1,k)$-cage.  Delete one arc from every vertex; the resulting graph has minimum undirected degree $r$, minimum directed out-degree $z$ and, as a subgraph of $G$, is obviously still $k$-geodetic. If $r = 0$, then the deleted arcs can be chosen such that the resulting subgraph $G'$ of $G$ contains a source vertex $z$, i.e. the in-degree of $z$ is zero; deleting $z$ does not decrease the geodetic girth or the minimum out-degree.  
	\end{proof}

	\section{Bounds on totally regular mixed graphs with small excess}\label{excesstotregular}
	
	The proof of the non-existence of mixed Moore graphs \cite{NguMilGim} uses an argument that admits of very useful generalisations.  We now present a counting argument that gives a new bound on the order of totally regular $(r,z,k;\epsilon )$-graphs.
	
	\begin{theorem}\label{bound for small excess}
		For $k \geq 3$, the excess $\epsilon $ of a totally regular $(r,z,k;+\epsilon )$-graph satisfies
		\[ \epsilon \geq \frac{r}{\phi }\Bigl[\frac{\lambda _1^{k-1}-1}{\lambda _1 - 1}-\frac{\lambda _2^{k-1}-1}{\lambda _2 - 1}\Bigr], \] 
		where
		\[ \phi = \sqrt{(r+z-1)^2+4z},  \]
		\[ \lambda _1 = \frac{1}{2}(r+z-1 + \phi) \]
		and
		\[ \lambda _2 = \frac{1}{2}(r+z-1 - \phi). \]
	\end{theorem}  
	\begin{proof}
		Let the vertex $x$ be the end-point of an arc in an undirected branch $T(u_i)$ of the Moore tree of depth $k$ based at a vertex $u$ such that $d(u,x) \leq k-1$. We will call such a vertex an \emph{arrow vertex} (with respect to $u$).  All undirected neighbours of the arrow vertex $x$ occur in $T(u_i)$, together with a single vertex of $Z^-(x)$.  As $G$ is totally regular, there are $z-1$ vertices of $U(x) \cup Z^-(x)$ that do not occur in $T(u_i)$.  Suppose that every vertex of $Z(u)$ can reach $x$ by a mixed path of length $\leq k$.  As $x$ cannot occur in the directed branches of $u$, it would then follow that each of the $z$ directed branches must contain a vertex in $U(x) \cup Z^- (x)$; however, as $r+1$ vertices of $U(x) \cup Z^-(x)$ already occur in $T(u_i)$, this means that a vertex is repeated in the Moore tree, which contradicts $k$-geodecity.  Therefore $x \in \cup_{u_i \in Z(u_i)}O(u_i)$.

		We now count the number of such arrow vertices $x$.  For $1 \leq t \leq k-1$, let $Z_t$ be the number of vertices in the undirected branches at Level $t$ in the Moore tree based at $u$ that are end-points of arcs emanating from Level $t-1$ and let $U_t$ be the number of vertices in the undirected branches at Level $t$ that are connected by an edge to Level $t-1$.  Obviously $U_1 = r, Z_1 = 0 $ and $Z_1 = rz$.  These numbers satisfy the recurrence relations
		\[ U_{t+1} = (r-1)U_t+rZ_t, Z_{t+1} = zU_t + zZ_t \]
		for $t \geq 1$.  It follows that
		\[ Z_{t+2} = zU_{t+1}+zZ_{t+1} = z((r-1)U_t+rZ_t)+zZ_{t+1}.\]
		Substituting using the second relation,
		\[ Z_{t+2} = zZ_{t+1}+rzZ_t+z(r-1)(1/z)(Z_{t+1}-zZ_t) = (r+z-1)Z_{t+1}+zZ_t.\]
		This second-order recurrence relation has characteristic equation 
		\[ \lambda ^2 - (r+z-1)\lambda - z = 0,\]
		with solutions $\lambda _1 , \lambda _2$ as given in the statement of the theorem.  Observe that the discriminant $\phi ^2 = (r+z-1)^2+4z$ is strictly positive, so $\lambda _1, \lambda _2$ are real and distinct.  It follows that
		\[ Z_t = A\lambda _1^t + B\lambda _2^t \]
		for $t \geq 1$ and some constants $A$ and $B$.  Substituting $Z_1 = 0, Z_2 = rz$, we obtain
		\[ Z_t = \frac{rz}{\phi }(\lambda _1^{t-1}-\lambda _2^{t-1})\]
		for $t \geq 1$.  Summing, we find that there are 
		\[ \sum_{i = 0}^{k-2}\frac{rz}{\phi }(\lambda _1^{i}-\lambda _2^{i}) = \frac{rz}{\phi }\left[\frac{\lambda _1^{k-1}-1}{\lambda _1 - 1}-\frac{\lambda _2^{k-1}-1}{\lambda _2 - 1}\right] \]
		such vertices.  As the union of the outlier sets of the vertices in $Z(u)$ contain a maximum of $z\epsilon $ distinct vertices between them, it follows that 
		\[ z\epsilon \geq \frac{rz}{\phi }\left[\frac{\lambda _1^{k-1}-1}{\lambda _1 - 1}-\frac{\lambda _2^{k-1}-1}{\lambda _2 - 1}\right] \]
		and the result follows.  
	\end{proof}
	Some values of the lower bound in Theorem \ref{bound for small excess} for $k = 4$ are displayed in Table \ref{fig:totregexcesstable}. We are not aware of any instance in which the bound of Theorem \ref{bound for small excess} is tight.  However, as we shall now demonstrate, it does yield a powerful result on mixed graphs with excess one.

	\begin{table}
		\begin{small}
			\begin{center}
				\begin{tabular}{| c||c|c| c| c| c |c |c| c| c| c| c }
					\hline
					
					$r$/$z$ & 1 & 2 & 3 & 4 & 5 & 6   \\
					\hline \hline
					1 & 2 & 3 & 4 & 5 & 6  & 7 \\
					2 & 6 & 8 & 10 & 12 & 14  & 16 \\
					3 & 12 & 15 & 18 & 21 & 24  & 27 \\
					4 & 20 & 24 & 28 & 32 & 36 & 40  \\ 
					5 & 30 & 35 & 40 & 45 & 50 & 55   \\
					6 & 42 & 48 & 54 & 60 & 66 & 72  \\ 
					7 & 56 & 63 & 70 & 77 & 84 & 91  \\ 
					8 & 72 & 80 & 88 & 96 & 104 & 112  \\ 
					9 & 90 & 99 & 108 & 117 & 126 & 135  \\
					10 & 110 & 120 & 130 & 140 & 150 & 160 \\  
					11 & 132 & 143 & 154 & 165 & 176 & 187   \\
					12 & 156 & 168 & 180 & 192 & 204 & 216  \\ 
					13 & 182 & 195 & 208 & 221 & 234 & 247  \\
					14 & 210 & 224 & 238 & 252 & 266 & 280  \\   
					15 & 240 & 255 & 270 & 285 & 300 & 315   \\
					\hline
					
				\end{tabular}
			\end{center}
		\end{small}
		\caption{Lower bound on the excess from Theorem \ref{bound for small excess} for $k = 4$}
		\label{fig:totregexcesstable}
	\end{table}

	\begin{corollary}\label{excess one large k implies k = 3}
		If $G$ is a totally regular $(r,z,k;+1)$-graph with $k \geq 3$, then $r = 1$ and $k = 3$.
	\end{corollary}

	\begin{theorem}\label{r = 1, epsilon = 1}
		There are no totally regular $(r,z,k;+1)$-graphs for $k \geq 3$.
	\end{theorem}
	\begin{proof}
		Let $G$ be a totally regular $(1,z,3;+1)$-graph.  For any vertex $u \in V(G)$ write $u^*$ for the undirected neighbour of $u$.  Let the adjacency matrices of $G, G^U$ and $G^Z$ be $A, A_U$ and $A_Z$ respectively.  Fix a vertex $u$ and draw the Moore tree rooted at $u$.  Examination of the Moore tree shows that there are two walks of length $\leq 3$ from $u$ to itself (the trivial walk $u$ and the walk $u \sim u^* \sim u$ of length two), two walks of length $\leq 3$ from $u$ to $u^*$ ($u \sim u^*$ and $u \sim u^* \sim u \sim u^*$), three walks of length $\leq 3$ from $u$ to any directed out-neighbour $v$ of $u$ ($u \rightarrow v, u \sim u^* \sim u \rightarrow v$ and $u \rightarrow v \sim v^* \sim v$) and unique walks of length $\leq 3$ from $u$ to the vertices at distance two and three from $u$.  It follows that
		\[ I+A+A^2+A^3 = I+J+A+A_Z-P,\]
		where $I$ is the $n \times n$ identity matrix, $J$ is the all-one matrix and $P_{vv'} = 1$ if $o(v)=v'$ and $0$ otherwise. As $G$ is totally regular, $J$ commutes with the left-hand side, $I$ and $A_Z$; therefore $JP = PJ$ and $o$ is a permutation.
		
		Take an edge $uu^*$.  The argument of Theorem \ref{bound for small excess} and the fact that $o$ is a permutation shows that $o(Z^+(u)) = Z^+(u^*)$ and $o(Z^+(u^*)) = Z^+(u)$.  Applying this result to an arbitrary directed in-neighbour $v$ of $u$, we see that there is a path $v \sim v^* \rightarrow o(u)$.  Let $w \in Z^+(o(u))$. A diagram of this situation is shown in Figure \ref{fig:permutation argument}.  There is a path of length three from $v$ to $w$, so $d(u,w) \geq 3$; in fact, since $o$ is a permutation, we have equality.  Since only $r+z-1$ in-neighbours of $w$ lie in the Moore tree rooted at $u$, it follows that $w$ must be the outlier of an out-neighbour of $u$.  Examining the Moore tree of depth three based at $v$, we see that if $w$ is an outlier of a vertex in $Z^+(u)$, then it would appear twice in the Moore tree rooted at $v$, once in the undirected $v^*$-branch and once in the $u$-branch in $Z^+(u^*)$, violating $3$-geodecity.  Therefore $v$ is the outlier of $u^*$; as the excess is one, $u^*$ has a unique outlier, so $z = 1$.  
		
		\begin{figure}\centering
			
			\begin{tikzpicture}[middlearrow=stealth,x=0.2mm,y=-0.2mm,inner sep=0.1mm,scale=2.90,
				thick,vertex/.style={circle,draw,minimum size=12,font=\tiny,fill=white},edge label/.style={fill=white}]
				\tiny
				\node at (50,100) [vertex] (w) {$w$};
				\node at (60,50) [vertex] (ou) {$o(u)$};
				\node at (70, 0) [vertex] (v*) {$v^*$};
				\node at (125,-50) [vertex] (v) {$v$};
				\node at (200,0) [vertex] (v0) {$u$};
				\node at (100,50) [vertex] (v1) {$u^*$};
				\node at (200,50) [vertex] (v2) {};
				\node at (300,50) [vertex] (v3) {};
				\node at (120,100) [vertex] (v4) {};
				\node at (180,100) [vertex] (v5) {};
				\node at (220,100) [vertex] (v6) {};
				\node at (200,100) [vertex] (v7) {};
				\node at (80,100) [vertex] (v9) {};
				\node at (280,100) [vertex] (v10) {};
				\node at (300,100) [vertex] (v11) {};
				\node at (320,100) [vertex] (v12) {};

				\node at (70,150) [vertex] (v13) {};
				\node at (80,150) [vertex] (v14) {};
				\node at (90,150) [vertex] (v15) {};
				
				\node at (110,150) [vertex] (v16) {};
				\node at (120,150) [vertex] (v17) {};
				\node at (130,150) [vertex] (v18) {};

				\node at (175,150) [vertex] (v19) {};
				\node at (185,150) [vertex] (v20) {};
				\node at (195,150) [vertex] (v21) {};
				\node at (200,150) [vertex] (v22) {};
				\node at (205,150) [vertex] (v23) {};
				\node at (215,150) [vertex] (v24) {};
				\node at (220,150) [vertex] (v25) {};
				\node at (225,150) [vertex] (v26) {};
				
				\node at (275,150) [vertex] (v27) {};
				\node at (285,150) [vertex] (v28) {};
				\node at (295,150) [vertex] (v29) {};
				\node at (300,150) [vertex] (v30) {};
				\node at (305,150) [vertex] (v31) {};
				\node at (315,150) [vertex] (v32) {};
				\node at (320,150) [vertex] (v33) {};
				\node at (325,150) [vertex] (v34) {};
				
				\path
				(ou) edge [middlearrow] (w)
				(v*) edge [middlearrow] (ou)
				(v) edge (v*)
				(v) edge [middlearrow] (v0)
				(v5) edge [middlearrow] (v19)
				(v5) edge [middlearrow] (v20)
				(v7) edge (v21)
				(v7) edge [middlearrow] (v22)
				(v7) edge [middlearrow]	(v23)
				(v6) edge (v24)
				(v6) edge [middlearrow] (v25)
				(v6) edge [middlearrow]	(v26)

				(v10) edge [middlearrow] (v27)
				(v10) edge [middlearrow] (v28)
				(v11) edge (v29)
				(v11) edge [middlearrow] (v30)
				(v11) edge [middlearrow]	(v31)
				(v12) edge (v32)
				(v12) edge [middlearrow] (v33)
				(v12) edge [middlearrow]	(v34)

				(v9) edge (v13)
				(v9) edge [middlearrow] (v14)
				(v9) edge [middlearrow] (v15)
				(v4) edge (v16)
				(v4) edge [middlearrow] (v17)
				(v4) edge [middlearrow] (v18)

				(v0) edge (v1)
				(v0) edge [middlearrow] (v2)
				(v0) edge [middlearrow] (v3)
				(v1) edge [middlearrow] (v9)
				(v1) edge [middlearrow] (v4)
				(v2) edge (v5)
				(v2) edge [middlearrow] (v6)
				(v2) edge [middlearrow] (v7)
				(v3) edge (v10)
				(v3) edge [middlearrow] (v11)
				(v3) edge [middlearrow] (v12)
				;

			\end{tikzpicture}
			\caption{Configuration for Theorem \ref{r = 1, epsilon = 1} for $z = 2$}
			\label{fig:permutation argument}
		\end{figure}
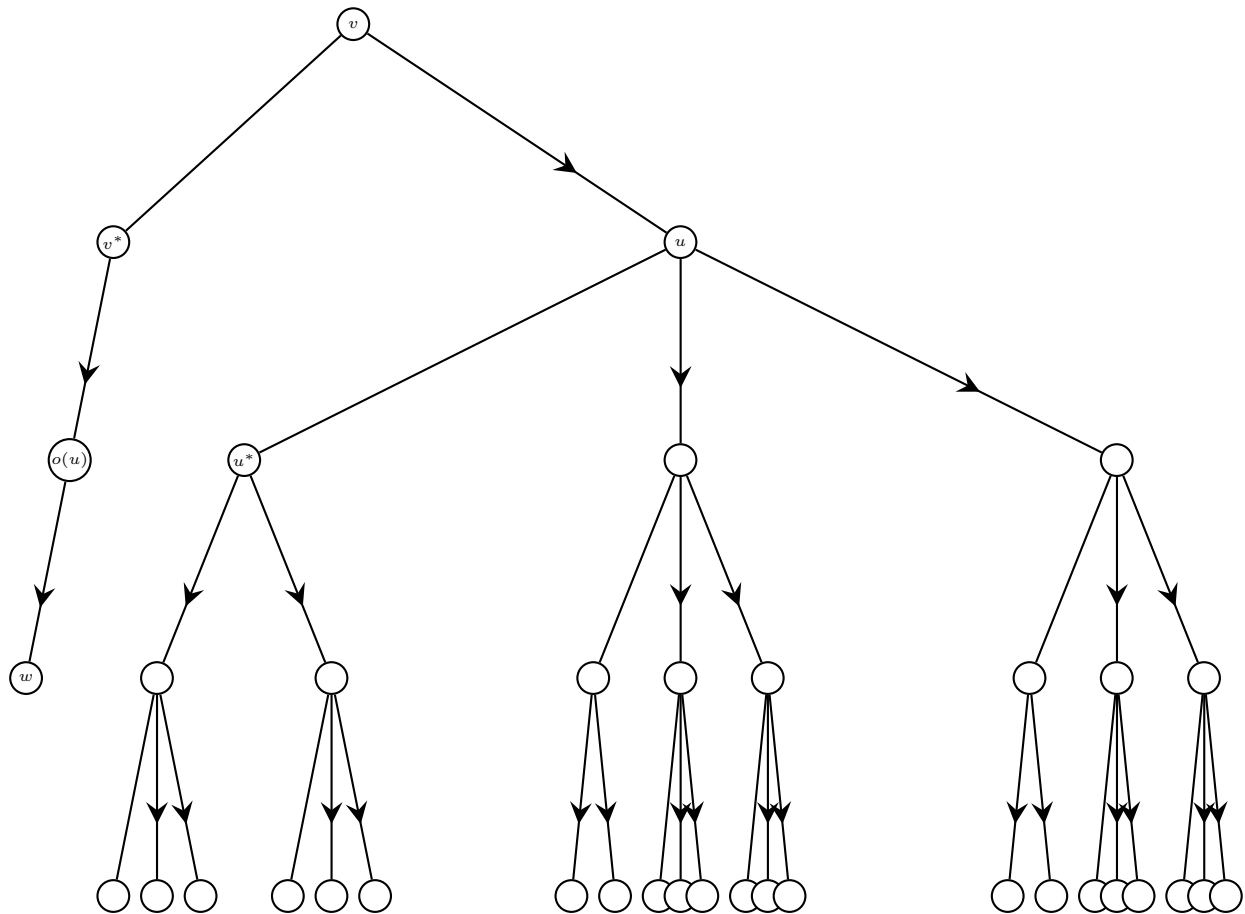
		
		We can dispose of the case $r = z = 1, k = 3$ by the argument of Theorem \ref{bound for small excess}.  Let $u_8 \sim a, u_8 \rightarrow b, u_9 \rightarrow c, u_{10} \sim d, u_{10} \rightarrow e$; see Figure \ref{fig:k3r1z1}.  Our argument shows that $o(u_2) = u_3$, so \[ \{ a,b,c,d,e\} = \{ u,u_1,u_6,u_7,u_{11}\} ,\]
		where $u_{11} = o(u)$. As the undirected neighbours of $u, u_1$ and $u_6$ are accounted for, $\{ b,c,e\} = \{ u,u_1,u_6\} $ and $\{ a,d\} = \{ u_7,u_{11}\} $. We have $\{ c,e\} \not = \{ u,u_1\} $ or there would be a repeat in the Moore tree rooted at $u_5$.  $u \not = b$ or else there would be paths $u_4 \sim u_2$ and $u_4 \rightarrow u_8 \rightarrow u \rightarrow u_2$, so $u \in \{ c,e\} $.  Thus $u_1 \not \in \{ c,e\} $, so $b = u_1$ and $\{ c,e \} = \{ u,u_6\} $.  By $3$-geodecity applied to $u_8$, $b = u_1$ implies that $a \not = u_7$, so $a = u_{11} = o(u)$ and hence $d = u_7$.  $e \not = u_6$, or $u_{10}$ would have two paths of length $\leq 3$ to $u_7$.  Therefore $c = u_6, e = u$.  
		
		Taking into account all adjacencies, it follows that there are three arcs from $\{ u_6,u_7,u_{11}\} $ to $\{ u_4,u_9,u_{11}\} $.  $u_{11} \not \rightarrow u_{11}$ and $u_{11} \not \rightarrow u_4$, or we would have $u_4 \rightarrow u_8 \sim u_{11} \rightarrow u_4$.  Hence $u_{11} \rightarrow u_9$.  $u_6 \not \rightarrow u_{11}$, or $u_9 \rightarrow u_6 \rightarrow u_{11} \rightarrow u_9$, so $u_6 \rightarrow u_4$ and $u_7 \rightarrow u_{11}$.  But now there are paths $u_1 \sim u \rightarrow u_2 \sim u_4$ and $u_1 \rightarrow u_3 \sim u_6 \rightarrow u_4$, contradicting $3$-geodecity.  As $G$ has even order, $G$ has excess $\epsilon \geq 3$.    
	\end{proof}
	
	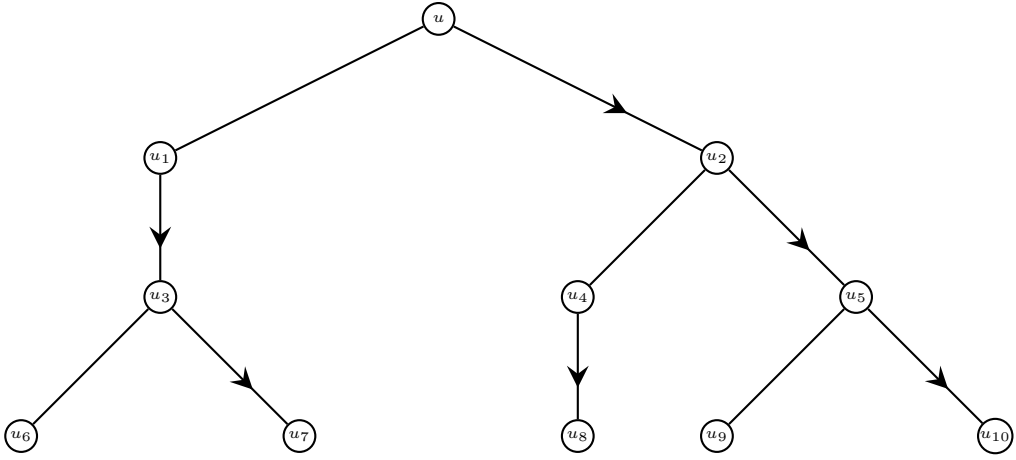
\begin{figure}\centering
		\begin{tikzpicture}[middlearrow=stealth,x=0.2mm,y=-0.2mm,inner sep=0.1mm,scale=1.85,
			thick,vertex/.style={circle,draw,minimum size=12,font=\tiny,fill=white},edge label/.style={fill=white}]
			
			\node at (0,0) [vertex] (v0) {$u$};
			\node at (-100,50) [vertex] (v1) {$u_1$};
			\node at (100,50) [vertex] (v2) {$u_2$};
			\node at (-100,100) [vertex] (v3) {$u_3$};
			\node at (50,100) [vertex] (v4) {$u_4$};
			\node at (150,100) [vertex] (v5) {$u_5$};
			\node at (-150,150) [vertex] (v6) {$u_6$};
			\node at (-50,150) [vertex] (v7) {$u_7$};
			\node at (50,150) [vertex] (v8) {$u_8$};
			\node at (100,150) [vertex] (v9) {$u_9$};
			\node at (200,150) [vertex] (v10) {$u_{10}$};
			
			\path
			(v0) edge (v1)
			(v0) edge [middlearrow] (v2)
			(v1) edge [middlearrow] (v3)
			(v2) edge (v4)
			(v2) edge [middlearrow] (v5)
			(v3) edge (v6)
			(v3) edge [middlearrow] (v7)
			(v4) edge [middlearrow] (v8)
			(v5) edge (v9)
			(v5) edge [middlearrow] (v10)
			
			;

		\end{tikzpicture}
		\caption{Moore tree for a $3$-geodetic mixed graph with $r = z = 1$}
		\label{fig:k3r1z1}
	\end{figure}
	It follows from Corollary \ref{excess one large k implies k = 3}, Theorem \ref{r = 1, epsilon = 1} and the results of \cite{TuiErs} that any $(r,z,k;+1)$-graph is either totally regular with $k = 2$, satisfying the conditions in Theorem \ref{spectral theorem}, or else $k \geq 3$, $z \geq 2$ and $G$ is not totally regular.
	
	We conclude this section with a result on the connection between outlier sets and automorphisms of mixed graphs with excess one. It is known that the outlier function of a $(d,k;+1)$-digraph $G$ is an automorphism if and only if $G$ is diregular \cite{Sil}.  The above results now allow us to extend this result to the more general mixed setting. 
	
	\begin{theorem}\label{automorphism}
		The outlier function of an $(r,z,k;+1)$-graph $G$ is an automorphism if and only if $G$ is totally regular.
	\end{theorem}
	\begin{proof}
		Suppose firstly that $G$ is not totally regular; recall that $G$ must be out-regular.  Let $v' \in S'$.  Suppose that $o$ is an automorphism. It follows that $o(v') \in S'$.  However, this implies that $o(v')$ has $> r+z$ in-neighbours distributed among the $r+z$ branches of the Moore tree based at $v'$, so that some out-neighbour of $v'$ has $\geq 2$  mixed paths to $o(v')$ with length $\leq k$. Thus if $o$ is an automorphism, then $G$ is totally regular.
		
		Now let $G$ be totally regular.  Let $k = 2$ and write $A$ for the adjacency matrix of $G$.  Then 
		\[ I + A + A^2 = J + rI - P,\]
		where $I$ is the $n \times n$ identity matrix, $J$ is the all-one matrix and $P_{uv} = 1$ if $o(u) = v$ and $0$ otherwise.  As $G$ is totally regular, both $I$ and $J$ commute with $A$.  Therefore $P$ commutes with $A$, so that $o$ is an automorphism.  There are no totally regular mixed graphs with excess one for $k \geq 3$ by Theorem \ref{r = 1, epsilon = 1}, so the proof is complete. 
	\end{proof}

	\section{Excess of mixed graphs that are not totally regular}\label{counting not totreg}
	
	We will now revisit the counting arguments used in the previous section to derive a bound in the more difficult context of mixed graphs that are not totally regular.  We will see that a bound for all mixed graphs, totally regular or not, can be achieved by relaxing the bound in Theorem \ref{bound for small excess} by a factor of $\frac{z}{2r+3z}$. We will need the following result from \cite{TuiErs}.  
	
	\begin{theorem}\cite{TuiErs}\label{TuiErs theorem}
		Any $(r,z,k;+1)$-graph must be totally regular if either $k = 2$ or $z = 1$.   
	\end{theorem}
	Using the new bound presented in the following theorem we will improve on Theorem \ref{TuiErs theorem} significantly.
	
	\begin{theorem}\label{bound for non totally regular}
		The excess of any $(r,z,k)$-cage satisfies 
		\[ \epsilon \geq \frac{rz}{(2r+3z)\phi }\Bigl[\frac{\lambda _1^{k-1}-1}{\lambda _1 - 1}-\frac{\lambda _2^{k-1}-1}{\lambda _2 - 1}\Bigr], \]
		where $\lambda _1, \lambda _2$ and $\phi $ are as defined in Theorem \ref{bound for small excess}.
	\end{theorem}  
	\begin{proof}
		Let $G$ be an $(r,z,k)$-cage. We can assume that the directed subgraph of $G$ is out-regular, by deleting some arcs if necessary.  Let the number of arrow vertices in the Moore tree of an out-regular $(r,z,k;+\epsilon )$-graph be $A(r,z,k)$.  By the calculation of Theorem \ref{bound for small excess} we know that \[ A(r,z,k) = \frac{rz}{\phi }\Bigl[\frac{\lambda _1^{k-1}-1}{\lambda _1 - 1}-\frac{\lambda _2^{k-1}-1}{\lambda _2 - 1}\Bigr].\] 
		We are therefore aiming to prove that 
		\[ \epsilon \geq \frac{1}{2r+3z}A(r,z,k).\]
		Clearly if there exists a vertex $u$ of $G$ with undirected degree $d(u) \geq r+1$, then the excess of $G$ would exceed $A(r,z,k)$ and hence also the bound of the theorem.  We can thus assume that $G$ is out-regular.
		Let the deficiency $\sigma ^-(v)$ of a vertex $v \in S$ be $z - d^-(v)$ and the surplus $\sigma ^+(v')$ of a vertex $v' \in S'$ be $d^-(v')-z$.  As $G$ is out-regular we have for the total deficiency $\sigma $
		\[ \sigma = \sum _{v\in S}\sigma ^-(v) = \sum _{v' \in S'}\sigma ^+(v').\]
		As each vertex in $S'$ contributes at least one to $\sigma $, we trivially have $\sigma \geq |S'|$.  We will now find an upper bound for $\sigma $ in terms of $r, z$ and $\epsilon $.
		
		Fix a vertex $u$ of $G$ and draw the Moore tree of depth $k$ rooted at $u$.  Write $U(u) = \{ u_1,u_2, \dots, u_r\} $. Let $v \in S$ have deficiency $\sigma ^-(v) = s$.  Suppose firstly that $d(u,v) \geq k$ (i.e. either $v$ lies at the bottom of the tree or $v \in O(u)$). Then $v$ can have in-neighbours in at most $r+z-s$ branches of the Moore tree and so lies in the outlier sets of at least $s$ members of $N^+(u)$.  
		
		Now suppose that either $u = v$ or $d(u,v) \leq k-1$ and $v$ lies in an undirected branch of the tree.  At most $z - s$ directed branches of the tree can contain in-neighbours of $v$ (in fact $z-s-1$ branches if $v$ is an arrow vertex), so again $v$ occurs at least $s$ times in the multiset $O(Z^+(u))$.
		
		Lastly we must consider the case that $v$ lies in a directed branch of the tree and $d(u,v) \leq k-1$.  Consider the Moore tree based at any $u_i \in U(u)$, say $u_1$.  $v$ lies in an undirected branch of this tree and so by our previous analysis $v$ occurs at least $s$ times in $O(N^+(u_1))$.  
		
		We have now dealt with all members of $S$.  Summing their deficiencies to find $\sigma $ we find that the elements of $S$ appear at least $\sigma $ times in the multiset $O(N^+(u))\cup O(N^+(u_1))$. As this multiset contains $(2r+2z)\epsilon $ elements, we conclude that
		\[ \sigma \leq (2r+2z)\epsilon .\]
		We now estimate the size of the set $S'$.  Again we consider the Moore tree rooted at $u$.  If an arrow vertex $x$ relative to $u$ lies in $V(G) - S'$, then $x$ cannot have an in-neighbour in every directed branch of the tree and so must be an outlier of at least one directed out-neighbour of $u$.  There are $z\epsilon $ elements in $O(Z^+(u))$, so it follows that at least $A(r,z,k)-z\epsilon $ of the arrow vertices must lie in $S'$.  Therefore \[ (2r+2z)\epsilon \geq \sigma \geq |S'| \geq A(r,z,k)-z\epsilon .\] 
		Rearranging we derive the inequality \[ \epsilon \geq \frac{1}{2r+3z}A(r,z,k).\]
		This proves the theorem. 
	\end{proof}
	This result now enables us to rule out the existence of mixed graphs with excess one for $k \geq 4$ and `most' values of $r$ and $z$ for $k = 3$.
	
	\begin{theorem}\label{excess one k geq 4}
		There are no $(r,z,k;+1)$-graphs for $k \geq 4$ or for $k = 3$, $r \geq 4$ and $z > \frac{2r}{r-3}$.
	\end{theorem}
	\begin{proof}
		Setting $\epsilon = 1$ in Theorem \ref{bound for non totally regular} shows that if $A(r,z,k) > 2r+3z$, then no $(r,z,k;+1)$-graph can exist.  If $k \geq 5$, then 
		\[ A(r,z,k) \geq A(r,z,5) = rz^3+2r^2z^2+r^3z-r^2z+rz.\] 
		If $z \geq 2$, then this expression obviously exceeds $2r+3z$, so let $z = 1$. Then by Theorem \ref{TuiErs theorem} $G$ must be totally regular; however, no such graphs exist by Theorem \ref{r = 1, epsilon = 1}.
		
		Let $k = 4$.  We have $A(r,z,4) = rz^2+zr^2$.  If $r \geq 2$ and $z \geq 2$, then $rz^2+zr^2 \geq 4r + 4z > 2r+3z$.  The result follows for $z = 1$ by Theorem \ref{TuiErs theorem} and Theorem \ref{r = 1, epsilon = 1}, so we can assume that $r = 1$.  We want to show that $z^2 + z > 3z + 2$, i.e. $z^2 - 2z - 2 > 0$.  This inequality holds for $z \geq 3$, so this leaves only the pair $(r,z) = (1,2)$ to deal with.  However in this case the Moore bound $M(1,2,4)$ is even, so that $G$ must have odd order.  However, $r = 1$ implies that $G$ has a perfect matching, so this is impossible.    	
		
		Finally let $k = 3$.  We have $A(r,z,3) = rz$, so $A(r,z,3) > 2r+3z$ if and only if $r \geq 4$ and $z > \frac{2r}{r-3}$.
	\end{proof}
	
	For $k = 3$ this leaves open the cases $r = 1,2,3$, $r = 4$ and $2 \leq z \leq 8$, $r = 5$ and $2 \leq z \leq 5$, $r = 6$ and $2 \leq z \leq 4$, $r = 7,8$ and $9$ and $2 \leq z \leq 3$ and $r \geq 10$ and $z = 2$. We can deal with the majority of these cases by a slightly more sophisticated method. 
	
	\begin{lemma}\label{in-degreeofv'}
		If $G$ is an $(r,z,k;+1)$-graph that is not totally regular, then every vertex $v' \in S'$ has directed in-degree $z+1$.  Therefore $\sigma = |S'|$.
	\end{lemma}
	\begin{proof}
		Consider the Moore tree rooted at $v' \in S'$.  Each branch of the tree can contain at most one in-neighbour of $v'$ by $k$-geodecity.  Therefore, as $v'$ has at least $r+z+1$ in-neighbours we conclude that each branch contains exactly one in-neighbour of $v'$ and $o(v') \in Z^-(v')$. Hence $v'$ has exactly $r+z+1$ in-neighbours.
	\end{proof}
	
	\begin{lemma}\label{nov'anoutlier}
		No $v' \in S'$ is an outlier.
	\end{lemma}
	\begin{proof}
		Assume for a contradiction that $G$ is an $(r,z,k;+1)$-graph in which $o(u) = v'$ for some $u \in V(G)$ and $v' \in S'$.  As $v'$ is the outlier of $u$, no in-neighbour of $v'$ can lie at distance less than $k$ from $u$.  By $k$-geodecity, we conclude that every branch of the Moore tree rooted at $u$ contains a unique in-neighbour of $v'$ at distance $k$ from $u$.  Therefore we must have $o(u) \in N^-(v')$ to account for the final in-neighbour of $v'$.  As $v' = o(u)$, this contradicts $k$-geodecity.
	\end{proof}
	
	\begin{lemma}\label{bound on sigma}
		For $k = 3$, if an $(r,z,3;+1)$-graph exists, then $z^2+z+r \geq \sigma \geq z+r$.	
	\end{lemma}
	\begin{proof}
		Let $G$ be an $(r,z,3;+1)$-graph. The Moore bound for $k = 3$ is \[ M(r,z,3) = r^3+z^3+3rz^2+3r^2z-r^2+z^2+r+z+1.\]
		The order of $G$ is $n = M(r,z,3)+1$.  The Moore bound for $k = 2$ is \[ M(r,z,2) = r^2+z^2+2rz+z+1.\]
		Fix some $v' \in S'$.  By Lemma \ref{nov'anoutlier}, every vertex of $G$ can reach $v'$ by a mixed path of length $\leq 3$.  We achieve a lower bound for the number of these vertices by assuming that $S \subseteq N^-(v')$.  Taking into account that $v'$ has exactly one extra directed in-neighbour by Lemma \ref{in-degreeofv'} and since all vertices of $T_{-3}(v')$ are distinct by $3$-geodecity we obtain the following inequality:
		\[ n = M(r,z,3)+1 \geq M(r,z,3)+M(r,z,2) -\sigma (1+r+z). \]
		Rearranging,
		\[ \sigma (1+r+z) \geq M(r,z,2)-1 = r^2+z^2+2rz+z. \]
		Multiplying out, it is easily seen that $\sigma \geq r+z.$
		Now we turn to the upper bound. Fix a vertex $u$ and draw the Moore tree based at $u$.  By the argument of Theorem \ref{bound for non totally regular}, we see that any vertex $v$ in $S$ that lies in $\{ u,o(u) \} \cup N^k(u)$ or any of the undirected branches of the tree must be an outlier of at least $\sigma ^-(v)$ vertices in $N^+(u)$.  Therefore these vertices between them contribute at most $r + z$ to the total $\sigma $. 
		
		Fix a directed out-neighbour $u^+$ of $u$ and consider the vertices in the Moore tree rooted at $u^+$ at distance $\leq 1$ from $u^+$.  Any vertex $v \in S$ belonging to this set will be an outlier of at least $\sigma ^-(v)$ vertices in $Z^+(u^+)$.  Between them such vertices can therefore contribute at most $z^2$ to the total $\sigma $.  Since we have now considered all vertices in $G$, the conclusion follows. 
	\end{proof}

	\begin{theorem}\label{excess one k = 3}
		There are no $(r,z,3;+1)$-graphs with $r \geq 2$.
	\end{theorem}
	\begin{proof}
		Suppose that $G$ is an $(r,z,3;+1)$-graph with $r > 1$. We know from Lemma \ref{bound on sigma} that $z^2+r+z \geq \sigma \geq r+z$, so we can write $\sigma = z^2+r+z-\alpha $, where $0 \leq \alpha \leq z^2$.   Fix an arbitrary vertex $u$ of $G$ and draw the Moore tree rooted at $u$. There are $rz$ arrow vertices in the tree relative to $u$, i.e. $rz$ vertices in the set $Z^+(U(u))$.  If any of the arrow vertices does not belong to $S'$, then it will be an outlier of a vertex in $Z^+(u)$.  It follows that at least $(r-1)z$ of the arrow vertices belong to $S'$.  Repeating this reasoning for each vertex in $N^+(u)$ and taking into account that the vertices of $Z^+(u)$ are arrow vertices relative to any vertex in $U(u)$, we see that there are at least
		\[ (r-1)z + (r-1)z + (r-1)(r-2)z + z^2(r-1) = (r-1)(z^2+rz)  \]
		vertices of $S'$ in the tree.  In fact, if we take $u$ to be an element of $S'$, a valid assumption by Theorem \ref{r = 1, epsilon = 1}, then we can actually deduce that
		\[ \sigma = z^2+r+z -\alpha = |S'| \geq (r-1)(z^2+rz)+1.\]
		Rearranging, we see that $\alpha $ must satisfy
		\[ \alpha \leq z^2+r+z-rz^2-r^2z+z^2+rz-1 = -zr^2-(z^2-z-1)r+(2z^2+z-1).\]
		If $r \geq 2$ and $z \geq 2$, then 
		\[ \alpha \leq -zr^2-(z^2-z-1)r+(2z^2+z-1) \leq -4z-2(z^2-z-1)+(2z^2+z-1) = -z+1 < 0, \]
		so it follows that we must have $r = 1$ and, considering the parity of the Moore bound, $z$ must be odd.
	\end{proof}
	
	By Theorems \ref{excess one k geq 4} and \ref{excess one k = 3} the only remaining open case left for $k \geq 3$ is the question of the existence of a non-totally regular $(1,z,3;+1)$-graph. We finally settle this outstanding problem. 
	\begin{theorem}\label{graphs with excess one are totreg}
		If $G$ is an $(r,z,k;+1)$-graph, then $k = 2$ and $G$ is totally regular.
	\end{theorem}
	\begin{proof}
		Suppose that $G$ is an $(r,z,k;+1)$-graph with $k \geq 3$. Then by Theorem \ref{excess one k = 3} we have $r = 1$, $k = 3$ and $z \geq 3$ is odd. Also $G$ is not totally regular by Theorem \ref{r = 1, epsilon = 1}. Fix a vertex $u$ of $G$. Let $u^*$ be the undirected neighbour of $u$ and $\{ u_1,u_2,\dots ,u_z\} $ be the set $Z^+(u)$ of directed out-neighbours of $u$. Draw the Moore tree of depth $3$ rooted at $u$. 
		
		By counting the in-neighbours of a vertex $v \in S$ that are available to lie in the directed branches of the tree, it can be seen that $v$ will be the outlier of at least $\sigma ^-(v)$ vertices of $N^+(u)$ unless $v$ lies in $U(Z^+(u))$, i.e. unless $v$ is the undirected neighbour of a directed out-neighbour of $u$. For example, if $v \in Z^+(u)$, then the vertices $u^*$ and $v$ can reach $v$ by mixed paths of length $\leq k$ and $v$ has two in-neighbours already appearing in the tree (one is $u$ and the other is $v^*$ at Level 2), so that $v$ has at most $z-\sigma^-(v)-1$ further in-neighbours that can lie in the remaining $z-1$ directed branches, so that $v$ is the outlier of at least $\sigma ^-(v)$ vertices in $N^+(u)$. Repeating this analysis for each position in the Moore tree implies the result. 
		
		However, if $v$ lies in $U(Z^+(u))$ then we can only say that it will be the outlier of at least $\sigma ^-(v)-1$ vertices of $N^+(u)$ (it can be reached by two vertices of $N^+(u)$ by $\leq k$-paths and has a further $z-\sigma ^-(v)$ in-neighbours available for the remaining $z-1$ directed branches). Observe also that if an arrow vertex in the Moore tree lies in $S$, then this vertex $v$ will be an outlier of at least $\sigma ^-(v)+1$ vertices of $Z^+(u)$.
		
		Summing the deficiencies of all the vertices in $S$ to get the total deficiency $\sigma $, we conclude that there are at most $2z+1$ vertices of $S$, for at most $z$ vertices of $S$ can lie in $U(Z^+(u))$ and every other vertex $v$ of $S$ is an outlier of at least $\sigma ^-(v)$ vertices in $o(N^+(u))$, which is a multiset with size $z+1$. We now make this estimate more precise. For any vertex $u$ of $G$ define $\rho (u) = |S \cap U(Z^+(u))|$. Also let $\rho _{\min } = \min \{ \rho (u) : u \in V(G)\} $. If $u$ is a vertex at which this minimum value $\rho _{ \min } $ is achieved, then as there are exactly $\rho _{\min } $ undirected neighbours of $Z^+(u)$ that lie in $S$, the total deficiency satisfies $\sigma \leq z+\rho _{\min } +1$. 
		
		Suppose that $\rho _{\min }\geq 1$. For any vertex $u$, the sets $U(Z^+(u))$, $U(Z^+(u^*))$ and $U(Z^+(u_i))$ for $1 \leq i \leq z$ are mutually disjoint and each contain at least $\rho _{\min }$ vertices of $S$, which are distinct by $3$-geodecity. Thus \begin{equation}\label{eqn:r=1k=3}
			(z+2)\rho _{\min } \leq |S| \leq \sigma \leq z+\rho _{\min }+1.
		\end{equation}
		Rearranging, we see that either $\rho _{\min } = 0$ or $\rho _{\min } = 1$. Suppose that $\rho _{\min } = 1$; then we have equality in Equation \ref{eqn:r=1k=3}, which implies that $|S| = z+2$ and $\rho (u^*) = \rho (u_i) = 1$ for $1 \leq i \leq z$. Then as $\rho (u) = 1$, there is a directed out-neighbour of $u$ (say $u_1$) such that $u_1^* \in S$. Applying the same reasoning to $u_1$, we conclude that each of the $z+2$ sets $U(Z^+(u_1))$, $U(Z^+(u_1^*))$ and $U(Z^+(u'))$, where $u'$ is any directed out-neighbour of $u_1$, each contain one element of $S$; however, including $u_1^*$, we see that there would be at least $z+3$ elements of $S$, a contradiction. 
		
		Thus $\rho _{\min } = 0$. Hence by Lemma \ref{bound on sigma} we have $\sigma = z+1$. As no vertices of $S$ lie in $U(Z^+(u))$, each of the $\geq z+1$ elements of $S$ is an outlier of at least $\sigma ^-(v)$ vertices of $N^+(u)$, so we must have $|S| = z+1$, each vertex in $S$ has directed in-degree $z-1$ and all outliers of vertices in $N^+(u)$ lie in $S$. If an arrow vertex in the tree belongs to $S$, then there would be at least $z+2$ outliers of vertices in $N^+(u)$, whereas if a vertex of $V(G)-(S \cup S')$ is an arrow vertex then it would be an outlier of a vertex in $N^+(u)$ in addition to the $z+1$ outliers accounted for by $S$; both situations are impossible. It follows that each of the arrow vertices in the Moore tree rooted at $u$, i.e. the set $Z^+(u^*)$, must lie in $S'$.
		
		As there are only $z+1$ vertices in $S$, at least one of the directed out-neighbours of $u$, say $u_1$, must have $\rho (u_1) \leq 1$. If $\rho (u_1) = 0$, then as above each of the arrow vertices in the Moore tree rooted at $u_1$ must lie in $S'$, so there would be at least $2z$ vertices of $S'$ in the Moore tree rooted at $u$. Suppose that $\rho (u_1) = 1$; then if $\geq 2$ of the arrow vertices in the Moore tree rooted at $u_1$ lie in $S$, i.e. if $|Z^+(u_1^*) \cap S| \geq 2$, then the $z$ vertices of $S$ outside of $U(Z^+(u_1))$ would account for at least $z+2$ outliers of the vertices in $N^+(u_1)$, which is impossible. Moreover, any arrow vertex of $u_1$ that does not lie in $S \cup S'$ will also be an outlier of a vertex in $Z^+(u_1)$, so at least $z-1$ of the arrow vertices in the tree rooted at $u_1$ are in $S'$.  However, together with the $z$ vertices of $Z^+(u^*)$, we have now produced at least $2z-1$ vertices of $S'$ in the Moore tree rooted at $u$, so that $z+1 = \sigma = |S'| \geq 2z-1$, which is impossible for $z \geq 3$.  \end{proof}
	This completes our classification of $k$-geodetic mixed graphs with excess one for $k \geq 3$. In \cite{TuiErs} the authors conjectured that any mixed graph with excess one is totally regular; Theorem \ref{graphs with excess one are totreg} proves this conjecture.

	\section{Bounds on totally regular mixed graphs with small defect}\label{improved Fiol bound}
	
	We now return to the degree/diameter problem for mixed graphs and extend the counting arguments from the previous section to deal with totally regular mixed graphs with small defect.  The first non-trivial bound for such graphs was derived in \cite{DalFioLop}, where it is shown that for a totally regular $(r,z,k;-\delta )$-graph with $k \geq 3$ the defect is bounded below by the undirected degree $r$. There is equality for $k = 3$ and hence the bound is tight.  We present a new upper bound on the order of totally regular $(1,1,k;-\delta )$-graphs that improves on the result of \cite{DalFioLop} for $k \geq 4$.
	
	Let $G$ be a totally regular mixed graph with undirected degree $r = 1$, directed degree $z = 1$ and diameter $k$.  We will denote the unique undirected neighbour of a vertex $v$ of $G$ by $v^*$, the directed in-neighbour by $v^-$ and the directed out-neighbour by $v^+$.  Since $r = 1$, $G$ contains a perfect matching and must have even order. 
	
	For any vertex $v$ of $G$ we make the further definition that $v^1 = (v^+)^*$, that is $v^1$ is the undirected neighbour of the directed out-neighbour of $v$.  We extend this definition as follows. We set $v^0 = v$ and by iteration define $v^s = (v^{s-1})^1$ for $s \geq 2$.  By analogy we specify that $v^{-1} = (v^*)^-$, so that $v^-$ is the directed in-neighbour of the undirected neighbour of $v$. Again we set iteratively $v^{-s} = (v^{-(s-1)})^-$. Notice that $(v^1)^{-1} = (v^{-1})^1 = v$ for all $v \in V(G)$.    
	
	We draw the Moore tree of $G$ of depth $k$ based at a vertex $u$ as indicated in Figure \ref{fig:rz1Mooretree}. In particular, if a vertex at Level $t \leq k-1$ of the tree has both an undirected neighbour and a directed out-neighbour below it at Level $t+1$ of the tree, then we will place the undirected neighbour on the left and label the vertices accordingly.  If $k \geq 3$, then there will be vertices repeated in the tree, so that a vertex of $G$ can receive distinct labels in the Moore tree; nevertheless, for counting purposes we will still distinguish between the position labels in the tree. The left-hand side branch beginning at $u_1$ is the \emph{undirected branch} and the right-hand side branch beginning at $u_2$ is the \emph{directed branch}.

	\begin{figure}\centering
		\begin{tikzpicture}[middlearrow=stealth,x=0.2mm,y=-0.2mm,inner sep=0.1mm,scale=1.35,
			thick,vertex/.style={circle,draw,minimum size=12,font=\tiny,fill=white},edge label/.style={fill=white}]
			\tiny
			\node at (0,0) [vertex] (u0) {$u_0$};
			
			\node at (-150,60) [vertex] (u1) {$u_1$};
			\node at (150,60) [vertex] (u2) {$u_2$};
			
			\node at (-150,120) [vertex] (u3) {$u_3$};
			\node at (50,120) [vertex] (u4) {$u_4$};
			\node at (250,120) [vertex] (u5) {$u_5$};

			\node at (-200,180) [vertex] (u6) {$u_6$};
			\node at (-100,180) [vertex] (u7) {$u_7$};
			\node at (50,180) [vertex] (u8) {$u_8$};
			\node at (200,180) [vertex] (u9) {$u_9$};
			\node at (300,180) [vertex] (u10) {$u_{10}$};

			\node at (-200,240) [vertex] (u11) {$u_{11}$};
			\node at (-125,240) [vertex] (u12) {$u_{12}$};
			\node at (-75,240) [vertex] (u13) {$u_{13}$};
			\node at (25,240) [vertex] (u14) {$u_{14}$};
			\node at (75,240) [vertex] (u15) {$u_{15}$};
			\node at (200,240) [vertex] (u16) {$u_{16}$};
			\node at (275,240) [vertex] (u17) {$u_{17}$};
			\node at (325,240) [vertex] (u18) {$u_{18}$};

			\node at (-215,300) [vertex] (u19) {$u_{19}$};
			\node at (-185,300) [vertex] (u20) {$u_{20}$};
			\node at (-125,300) [vertex] (u21) {$u_{21}$};
			\node at (-90,300) [vertex] (u22) {$u_{22}$};
			\node at (-60,300) [vertex] (u23) {$u_{23}$};
			\node at (25,300) [vertex] (u24) {$u_{24}$};
			\node at (60,300) [vertex] (u25) {$u_{25}$};
			\node at (90,300) [vertex] (u26) {$u_{26}$};
			\node at (185,300) [vertex] (u27) {$u_{27}$};
			\node at (215,300) [vertex] (u28) {$u_{28}$};
			\node at (275,300) [vertex] (u29) {$u_{29}$};
			\node at (310,300) [vertex] (u30) {$u_{30}$};
			\node at (340,300) [vertex] (u31) {$u_{31}$};

			\path
			(u0) edge (u1)
			(u0) edge [middlearrow] (u2)
			
			(u1) edge [middlearrow] (u3)
			(u2) edge (u4)
			(u2) edge [middlearrow] (u5)
			
			(u3) edge (u6)
			(u3) edge [middlearrow] (u7)
			(u4) edge [middlearrow] (u8)
			(u5) edge (u9)
			(u5) edge [middlearrow] (u10)
			
			(u7) edge (u12)
			(u8) edge (u14)
			(u10) edge (u17)
			(u6) edge [middlearrow] (u11)
			(u7) edge [middlearrow] (u13)
			(u8) edge [middlearrow] (u15)
			(u9) edge [middlearrow] (u16)
			(u10) edge [middlearrow] (u18)

			(u11) edge (u19)
			(u11) edge [middlearrow] (u20)
			(u12) edge [middlearrow] (u21)
			(u13) edge (u22)
			(u13) edge [middlearrow] (u23)
			(u14) edge [middlearrow] (u24)
			(u15) edge (u25)
			(u15) edge [middlearrow] (u26)
			(u16) edge (u27)
			(u16) edge [middlearrow] (u28)
			(u17) edge [middlearrow] (u29)
			(u18) edge (u30)
			(u18) edge [middlearrow] (u31)
			;
		\end{tikzpicture}
		\caption{The Moore tree for $r = z = 1$ and $k = 5$.}
		\label{fig:rz1Mooretree}
	\end{figure}
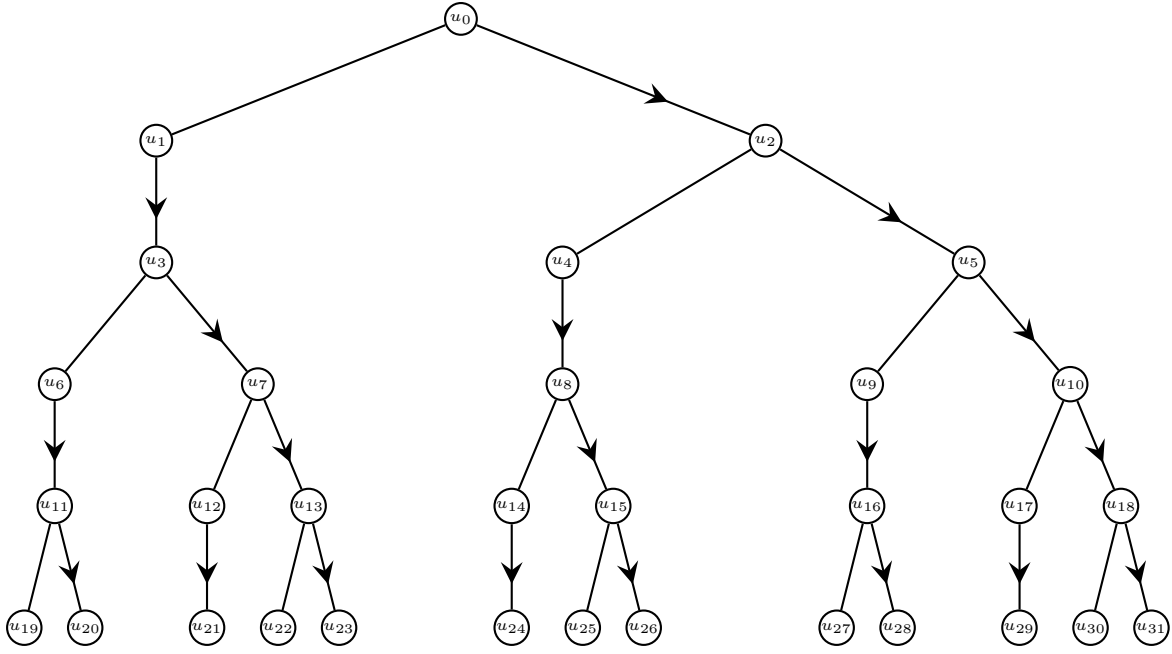
	
	To reiterate, an arrow vertex in the Moore tree of $G$ rooted at $u$ is a vertex $x$ at a Level $t$, $2 \leq t \leq k-1$, of the tree in the undirected branch such that $x$ appears as the terminal vertex of an arc with its initial vertex at Level $t-1$. Unlike the $k$-geodetic case, arrow vertices can be equal in $G$ or be equal to a vertex in the directed branch; therefore we will slightly abuse the term `arrow vertex' by associating it, not with a vertex of $G$, but with a position or label in the tree.
	
	Consider an arrow vertex $x$ at Level $t$ of the Moore tree. Its directed in-neighbour $x^-$ appears at Level $t-1$ and its undirected neighbour $x^*$ at Level $t+1$, so that the entire in-neighbourhood $N^-(x) = \{ x^-,x^*\} $ is also contained in the undirected branch of the Moore tree. As $G$ has diameter $k$, $u_2$ must be able to reach $x$ by a mixed path of length $\leq k$, so it follows that at least one of $x^-,x^*$ also appears in the directed branch of $G$. For every such occurrence there will be an additional repeat of $u_0$, so that we can bound the defect $\delta $ from below by counting the smallest possible number of positions in the undirected branch such that for every arrow vertex $x$ either $x^*$ or $x^-$ lies in one of these positions.  We will call such a set of positions a \emph{transversal} of the undirected branch. 
	
	We will now focus on the undirected branch of the Moore tree. The undirected branch of a Moore tree of depth 8 is shown in Figure \ref{fig:undirectedk8}. For convenience we use a different labelling of the undirected branch; for example, vertex $1$ corresponds to $u_1$ in Figure \ref{fig:rz1Mooretree}, $2$ to $u_3$, $3$ to $u_6$, $5$ to $u_{11}$, etc. For the moment we ignore the complication that a vertex of $G$ could appear multiple times as an arrow vertex in this tree. Under this assumption we will show that $\delta $ is bounded from below by the size of a minimum transversal of the Moore tree.

	\begin{figure}\centering
		\begin{tikzpicture}[middlearrow=stealth,x=0.2mm,y=-0.2mm,inner sep=0.1mm,scale=1.35,
			thick,vertex/.style={circle,draw,minimum size=12,font=\tiny,fill=white},edge label/.style={fill=white}]
			\tiny

			\node at (0,0) [vertex] (1) {1};

			\node at (0,60) [vertex] (2) {2};

			\node at (-150,120) [vertex] (3) {3};
			\node at (150,120) [vertex] (4) {4};

			\node at (-150,180) [vertex] (5) {5};
			\node at (50,180) [vertex] (6) {6};
			\node at (250,180) [vertex] (7) {7};

			\node at (-200,240) [vertex] (8) {8};
			\node at (-100,240) [vertex] (9) {9};
			\node at (50,240) [vertex] (10) {10};
			\node at (200,240) [vertex] (11) {11};
			\node at (300,240) [vertex] (12) {12};
			
			\node at (-200,300) [vertex] (13) {13};
			\node at (-130,300) [vertex] (14) {14};
			\node at (-70,300) [vertex] (15) {15};
			\node at (20,300) [vertex] (16) {16};
			\node at (80,300) [vertex] (17) {17};
			\node at (200,300) [vertex] (18) {18};
			\node at (270,300) [vertex] (19) {19};
			\node at (330,300) [vertex] (20) {20};	
			
			\node at (-220,360) [vertex] (21) {21};
			\node at (-180,360) [vertex] (22) {22};					
			\node at (-130,360) [vertex] (23) {23};	
			\node at (-90,360) [vertex] (24) {24};
			\node at (-50,360) [vertex] (25) {25};
			\node at (20,360) [vertex] (26) {26};
			\node at (60,360) [vertex] (27) {27};	 
			\node at (100,360) [vertex] (28) {28};
			\node at (180,360) [vertex] (29) {29};	
			\node at (220,360) [vertex] (30) {30};
			\node at (270,360) [vertex] (31) {31};
			\node at (310,360) [vertex] (32) {32};
			\node at (350,360) [vertex] (33) {33};	        	       	    	    	    
			
			\node at (-220,420) [vertex] (34) {34};
			\node at (-190,420) [vertex] (35) {35};
			\node at (-170,420) [vertex] (36) {36};
			\node at (-140,420) [vertex] (37) {37};	
			\node at (-120,420) [vertex] (38) {38};
			\node at (-90,420) [vertex] (39) {39};
			\node at (-60,420) [vertex] (40) {40};
			\node at (-40,420) [vertex] (41) {41};
			\node at (10,420) [vertex] (42) {42};
			\node at (30,420) [vertex] (43) {43};
			\node at (60,420) [vertex] (44) {44};
			\node at (90,420) [vertex] (45) {45};
			\node at (110,420) [vertex] (46) {46};
			\node at (180,420) [vertex] (47) {47};
			\node at (210,420) [vertex] (48) {48};
			\node at (230,420) [vertex] (49) {49};
			\node at (260,420) [vertex] (50) {50};
			\node at (280,420) [vertex] (51) {51}; 
			\node at (310,420) [vertex] (52) {52};
			\node at (340,420) [vertex] (53) {53};
			\node at (360,420) [vertex] (54) {54};
			\path
			
			(21) edge [middlearrow] (34)
			(22) edge (35)
			(22) edge [middlearrow] (36)
			(23) edge (37)
			(23) edge [middlearrow] (38)
			(24) edge [middlearrow] (39)
			(25) edge (40)
			(25) edge [middlearrow] (41)
			(26) edge (42)
			(26) edge [middlearrow] (43)
			(27) edge [middlearrow] (44)
			(28) edge (45)
			(28) edge [middlearrow] (46)
			(29) edge [middlearrow] (47)
			(30) edge (48)
			(30) edge [middlearrow] (49)
			(31) edge (50)
			(31) edge [middlearrow] (51)
			(32) edge [middlearrow] (52)
			(33) edge (53)
			(33) edge [middlearrow] (54)
			
			(13) edge (21)
			(15) edge (24)  
			(17) edge (27)
			(18) edge (29)
			(20) edge (32)              
			(13) edge [middlearrow] (22)
			(14) edge [middlearrow] (23)
			(15) edge [middlearrow] (25)
			(16) edge [middlearrow] (26)
			(17) edge [middlearrow] (28)
			(18) edge [middlearrow] (30)
			(19) edge [middlearrow] (31)
			(20) edge [middlearrow] (33)							
			
			(8) edge [middlearrow] (13)
			(9) edge [middlearrow] (15)
			(10) edge [middlearrow] (17)
			(11) edge [middlearrow] (18) 
			(12) edge [middlearrow] (20)  
			(9) edge (14)
			(10) edge (16)
			(12) edge (19)

			(1) edge [middlearrow] (2)

			(2) edge (3)
			(2) edge [middlearrow] (4)

			(3) edge [middlearrow] (5)
			
			(4) edge (6)
			(4) edge [middlearrow] (7)

			(5) edge (8)
			(5) edge [middlearrow] (9)
			(6) edge [middlearrow] (10)
			(7) edge (11)
			(7) edge [middlearrow] (12)
			
			;
		\end{tikzpicture}
		\caption{The undirected branch for $k = 8$}
		\label{fig:undirectedk8}
	\end{figure}
	
	Consider an arrow vertex $x$ at Level $t$ of the tree, where $2 \leq t \leq k-1$. In the undirected branch shown in Figure \ref{fig:undirectedk8} these are vertices $2$, $4$, $5$, $7$, $9$, $10$, $12$, $13$, $15$, $17$, $18$, $20$, $22$, $23$, $25$, $26$, $28$, $30$, $31$ and $33$. As already noted, either the undirected neighbour $x^*$ or the directed in-neighbour $x^-$ of $x$ must occur in the directed branch of the Moore tree, and each such occurrence counts towards the number of repeats of the root vertex $u$ of the tree.  However, the in-neighbourhoods of the arrow vertices overlap; for example, the vertex $8$ is an in-neighbour both of the vertex $5$ and the vertex $13$. We will partition the positions in the undirected branch of the Moore tree corresponding to vertices in the in-neighbourhoods of the arrow vertices into \emph{chains}. 
	
	A chain is a maximal string of vertices in the undirected branch of the Moore tree of the form $v=v^0,v^1,v^2,v^3, \ldots $, where $v$ is an in-neighbour of an arrow vertex. If $v$ is at Level $t \leq k-2$, then $v^2$ is at Level $t+2$.  For example $1,3,8,21$ is a chain which we have labelled (a) in Figure \ref{fig:undirectedk8}.  Every in-neighbourhood of an arrow vertex is contained in a unique chain.  Every arrow vertex at Level $t$, where $2 \leq t \leq k-2$, is the beginning of a chain, as is the vertex $1$.  Conversely, by iterating the $^-$ operation on an in-neighbour of an arrow vertex, i.e. considering the sequence of vertices $v,v^{-1},v^{-2}, \dots $, we see that every chain begins either at $1$ or an arrow vertex at Level $t \leq k-2$. This decomposition is displayed for $k = 8$ in Figure \ref{fig:chains}.

	\begin{figure}\centering
		\begin{tikzpicture}[middlearrow=stealth,x=0.2mm,y=-0.2mm,inner sep=0.1mm,scale=1.35,
			thick,vertex/.style={circle,draw,minimum size=12,font=\tiny,fill=white},edge label/.style={fill=white}]
			\tiny

			\node at (0,0) [vertex,label=above left:(a)] (1) {1};

			\node at (0,60) [vertex,label=above left:(b)] (2) {2};

			\node at (-150,120) [vertex,label=above left:(a)] (3) {3};
			\node at (150,120) [vertex,label=above left:(c)] (4) {4};

			\node at (-150,180) [vertex,label=above left:(d)] (5) {5};
			\node at (50,180) [vertex,label=above left:(b)] (6) {6};
			\node at (250,180) [vertex,label=above left:(e)] (7) {7};

			\node at (-200,240) [vertex,label=above left:(a)] (8) {8};
			\node at (-100,240) [vertex,label=above left:(f)] (9) {9};
			\node at (50,240) [vertex,label=above left:(g)] (10) {10};
			\node at (200,240) [vertex,label=above left:(c)] (11) {11};
			\node at (300,240) [vertex,label=above left:(h)] (12) {12};
			
			\node at (-200,300) [vertex,label=above left:(i)] (13) {13};
			\node at (-130,300) [vertex,label=above left:(d)] (14) {14};
			\node at (-70,300) [vertex,label=above left:(k)] (15) {15};
			\node at (20,300) [vertex,label=above left:(b)] (16) {16};
			\node at (80,300) [vertex,label=above left:(j)] (17) {17};
			\node at (200,300) [vertex,label=above left:(l)] (18) {18};
			\node at (270,300) [vertex,label=above left:(e)] (19) {19};
			\node at (330,300) [vertex,label=above left:(m)] (20) {20};	
			
			\node at (-220,360) [vertex,label=above left:(a)] (21) {21};
			\node at (-180,360) [vertex] (22) {22};					
			\node at (-130,360) [vertex] (23) {23};	
			\node at (-90,360) [vertex,label=above left:(f)] (24) {24};
			\node at (-50,360) [vertex] (25) {25};
			\node at (20,360) [vertex] (26) {26};
			\node at (60,360) [vertex,label=above left:(g)] (27) {27};	 
			\node at (100,360) [vertex] (28) {28};
			\node at (180,360) [vertex,label=above left:(c)] (29) {29};	
			\node at (220,360) [vertex] (30) {30};
			\node at (270,360) [vertex] (31) {31};
			\node at (310,360) [vertex,label=above left:(h)] (32) {32};
			\node at (350,360) [vertex] (33) {33};	        	       	    	    	    
			
			\node at (-220,420) [vertex] (34) {34};
			\node at (-190,420) [vertex,label=above left:(i)] (35) {35};
			\node at (-170,420) [vertex] (36) {36};
			\node at (-140,420) [vertex,label=above left:(d)] (37) {37};	
			\node at (-120,420) [vertex] (38) {38};
			\node at (-90,420) [vertex] (39) {39};
			\node at (-60,420) [vertex,label=above left:(k)] (40) {40};
			\node at (-40,420) [vertex] (41) {41};
			\node at (10,420) [vertex,label=above left:(b)] (42) {42};
			\node at (30,420) [vertex] (43) {43};
			\node at (60,420) [vertex] (44) {44};
			\node at (90,420) [vertex,label=above left:(j)] (45) {45};
			\node at (110,420) [vertex] (46) {46};
			\node at (180,420) [vertex] (47) {47};
			\node at (210,420) [vertex,label=above left:(l)] (48) {48};
			\node at (230,420) [vertex] (49) {49};
			\node at (260,420) [vertex,label=above left:(e)] (50) {50};
			\node at (280,420) [vertex] (51) {51}; 
			\node at (310,420) [vertex] (52) {52};
			\node at (340,420) [vertex,label=above left:(m)] (53) {53};
			\node at (360,420) [vertex] (54) {54};
			\path
			
			(21) edge [middlearrow] (34)
			(22) edge (35)
			(22) edge [middlearrow] (36)
			(23) edge (37)
			(23) edge [middlearrow] (38)
			(24) edge [middlearrow] (39)
			(25) edge (40)
			(25) edge [middlearrow] (41)
			(26) edge (42)
			(26) edge [middlearrow] (43)
			(27) edge [middlearrow] (44)
			(28) edge (45)
			(28) edge [middlearrow] (46)
			(29) edge [middlearrow] (47)
			(30) edge (48)
			(30) edge [middlearrow] (49)
			(31) edge (50)
			(31) edge [middlearrow] (51)
			(32) edge [middlearrow] (52)
			(33) edge (53)
			(33) edge [middlearrow] (54)
			
			(13) edge (21)
			(15) edge (24)  
			(17) edge (27)
			(18) edge (29)
			(20) edge (32)              
			(13) edge [middlearrow] (22)
			(14) edge [middlearrow] (23)
			(15) edge [middlearrow] (25)
			(16) edge [middlearrow] (26)
			(17) edge [middlearrow] (28)
			(18) edge [middlearrow] (30)
			(19) edge [middlearrow] (31)
			(20) edge [middlearrow] (33)							
			
			(8) edge [middlearrow] (13)
			(9) edge [middlearrow] (15)
			(10) edge [middlearrow] (17)
			(11) edge [middlearrow] (18) 
			(12) edge [middlearrow] (20)  
			(9) edge (14)
			(10) edge (16)
			(12) edge (19)

			(1) edge [middlearrow] (2)

			(2) edge (3)
			(2) edge [middlearrow] (4)

			(3) edge [middlearrow] (5)
			
			(4) edge (6)
			(4) edge [middlearrow] (7)

			(5) edge (8)
			(5) edge [middlearrow] (9)
			(6) edge [middlearrow] (10)
			(7) edge (11)
			(7) edge [middlearrow] (12)
			
			;
		\end{tikzpicture}
		\caption{The chain decomposition for $k = 8$}
		\label{fig:chains}
	\end{figure}
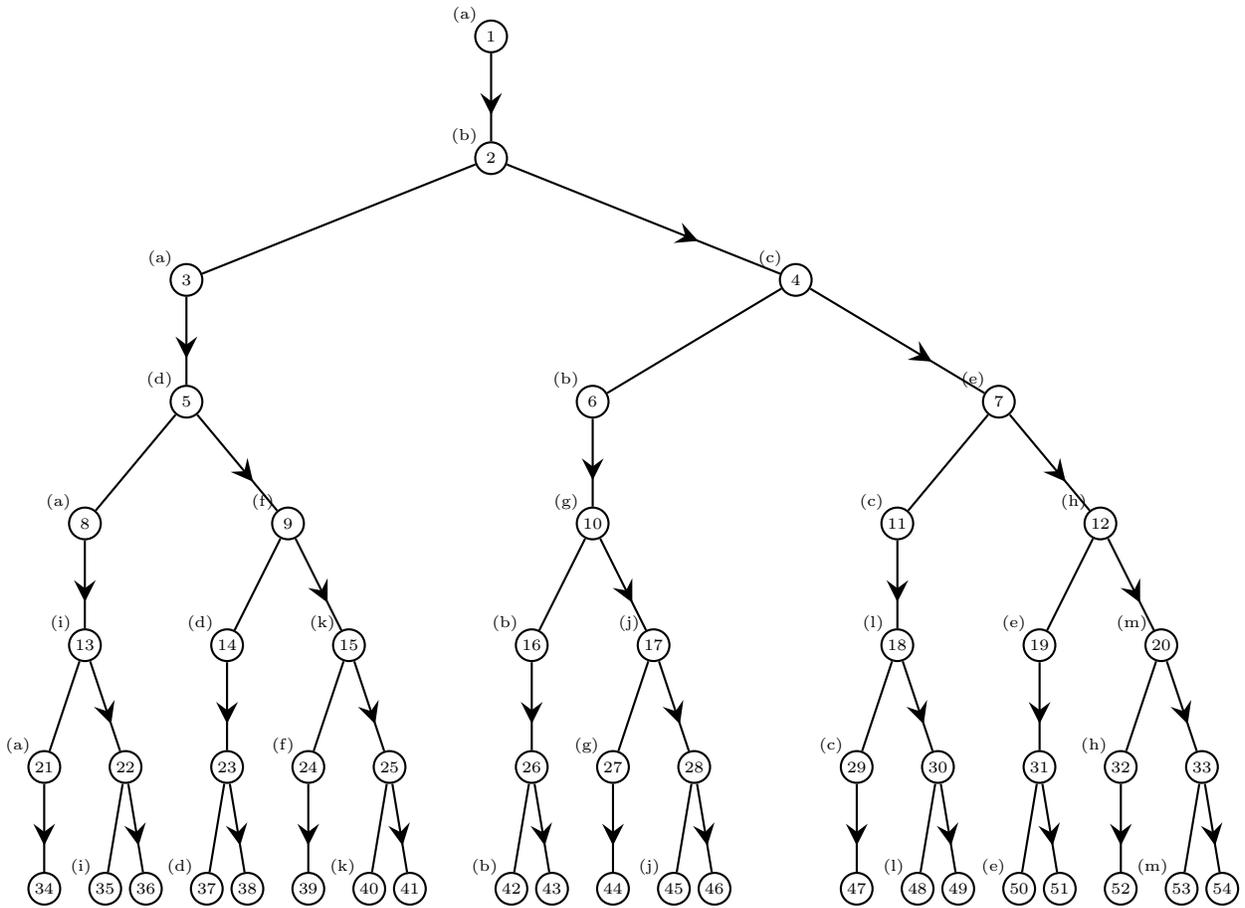
	
	We will call the number of vertices (i.e. positions in the Moore tree) in a chain $v,v^1,v^2,\dots $ the \emph{length} of the chain. For example, for $k = 8$ the chain $1,3,8,21$ has length $4$. Let $C$ be a chain of length $\ell $. Any pair of consecutive vertices in $C$ is the in-neighbourhood of an arrow vertex, so at least one of them must appear in the directed branch of the Moore tree. As any vertex in $C$ is contained in two pairs of consecutive vertices of the chain, it follows that the smallest transversal of $C$, i.e. the smallest number of vertices in the Moore tree that intersect every in-neighbourhood of arrow vertices that is contained in the chain, is $\lceil \frac{\ell }{3}\rceil $ (this follows from the domination number of the path \cite{ChaLesZha}). 
	
	The number of chains beginning at Level $t$ of the tree, where $2 \leq t \leq k-2$, is equal to the number of arrow vertices at Level $t$. From the calculation of Theorem \ref{bound for small excess} we know that this number is \[ Z_t = \frac{1}{2^{t-1}\sqrt {5} }((1+\sqrt {5})^{t-1}-(1-\sqrt {5})^{t-1}).\]
	The first vertex $1$ of the undirected branch is also the first vertex of a chain. We therefore define $Z'_t = 1$ for $t = 1$ and $Z'_t = Z_t$ for $2 \leq t \leq k-1$.  The length of a chain beginning at Level $t$ is $\ell (t) = 1+\lfloor \frac{k-t}{2} \rfloor $.  It follows from our argument that the smallest transversal of the undirected branch of the Moore tree has size \[ \sum^{k-2}_{t=1} Z'_t\left\lceil \frac{1}{3} + \frac{1}{3}\left\lfloor \frac{k-t}{2} \right\rfloor \right\rceil  .\] 
	This expression gives a lower bound for the number of positions in the undirected branch of the Moore tree that are occupied by vertices that also appear in the directed branch. It could happen that these positions in the undirected branch are actually occupied by the same vertex, which would reduce the number of vertices that would have to be repeated in the directed branch. 
	
	However, it is easily seen that this does not affect our lower bound for the defect. Let $T$ be the transversal of the undirected branch that is repeated in the directed branch of a largest $(r,z,k;-\delta )$-graph. If $s$ positions of $T$ are occupied by the same vertex $v$, then $v$ occurs at least once in the directed branch of the Moore tree, but is also repeated at least $s-1$ times in the undirected branch, so that this set of $s$ positions nevertheless contributes at least $s$ to the total defect $\delta $. We therefore have proved the following theorem.   
	
	\begin{theorem}
		Any totally regular $(1,1,k;-\delta)$-graph has defect
		\[ \delta \geq \sum^{k-2}_{t=1} Z'_t\left\lceil \frac{1}{3} + \frac{1}{3}\left\lfloor \frac{k-t}{2} \right\rfloor \right\rceil  \] 
		for $k \geq 3$.	
	\end{theorem}

	\section{Directed and mixed cages}\label{directed and mixed cages}
	We summarise here the results of a computer search for the smallest possible digraphs of given $d,k$ and mixed graphs for certain values of $r,z,k$. Such searches quickly become computationally infeasible as the order of the graphs grows. In many cases we can obtain a useful upper bound on the order of cages by restricting the search space to Cayley graphs; thus we present tables of the smallest Cayley graphs separately from tables of smallest general graphs.
	
	Recall that a Cayley graph $\Cay(G,S)$ of a group $G$ and subset $S\subseteq G$ has vertex set the elements of the group $G$, and a (directed) arc from $x$ to $xs$ for every $s\in S$. If the  set $S$ contains involutions or inverse pairs, then the resulting directed 2-cycles in the Cayley graph are considered to be undirected edges. Thus a mixed Cayley graph of order $n$, undirected degree $r$ and directed degree $z$ is constructed from a group $G$ of order $n$, together with a set $S$ such that $S$ contains exactly $r$ elements whose inverse is also in $S$, and $z$ elements whose inverse is not. To ensure the resulting Cayley graph is connected, we insist that $\langle S\rangle=G$.
	
	The Cayley graph search was carried out using \texttt{GAP}~\cite{GAP4} and proceeded by examining each possible group in increasing order, starting from the Moore bound for given $r,z,k$. The geodecity of such a graph is then the largest value of $k$ for which all possible words of length $\leq k$ in the generating set $S$ have different values. (We consider only reduced words, i.e. words in which a generator is not immediately followed by its inverse.) It is well known that if $\phi$ is an automorphism of the group $G$, then the Cayley graphs $\Cay(G,S)$ and $\Cay(G,\phi(S))$ are isomorphic. This provides a very useful means to cut down the search space for Cayley graphs, since only orbit representatives of possible generating sets need be considered.
	
	The general graph search was carried out using a bespoke C program, the output of which was tested against the Cayley graph search to ensure correct functioning. The program proceeds by starting with a Moore tree for given values of $r$ and $z$, then adding vertices and arcs and/or edges to obtain a graph of order $n$. As each arc/edge is added, the graph obtained is tested to ensure it still has geodecity at least $k$; if not, the search backtracks and tries another arc/edge.
	
	Selected output graphs from the search are illustrated in Section~\ref{sec:figs}.
	\subsection{Digraphs}
	Table~\ref{tab:caydig} shows the results of the Cayley digraph search. For $d=2$ we were able to complete the search for all groups of order less than 1024, so the results presented are known to be optimal. For higher degrees the search space becomes increasingly large, and so for degrees 5 and 6 we were only able to search far enough to determine the smallest digraphs of geodecity 2.
	\begin{table}[h]\centering
		\begin{tabular}{|lll|lll|}
			\hline
			$d$ & $k$ & $M$ & $n$ & $\epsilon$ & Group \\
			\hline
			2 & 2 & 7 & 12 & 5 & $\mathrm{Dic}_{12}$ \\
			& 3 & 15 & 20 & 5 & $\mathrm{AGL}(1,5)$ \\
			& 4 & 31 & 54 & 23 & $\Z_9\rtimes\Z_6$ \\
			& 5 & 63 & 136 & 73 & $Z_{17}\rtimes\Z_8$ \\
			& 6 & 127 & 330 & 203 & $Z_3\rtimes(\Z_{11}\rtimes\Z_{10})$ \\
			& 7 & 255 & 720 & 465 & $\mathrm{PGL}(2,9)$ \\
			\hline
			3 & 2 & 13 & 20 & 7 & $\mathrm{AGL}(1,5)$ \\
			& 3 & 40 & 72 & 32 & $S_3\wr S_2$ \\
			& 4 & 121 & 320 & 199 & $(((\Z_2\times Q_8) \rtimes \Z_2) \rtimes \Z_5) \rtimes \Z_2$ \\
			\hline 
			4 & 2 & 21 & 27 & 6 & $(\Z_3\times\Z_3)\rtimes\Z_3$ \\
			& 3 & 85 & 136 & 51 & $\Z_{17}\rtimes\Z_8$ \\
			\hline 
			5 & 2 & 31 & 42 & 11 & $\mathrm{AGL}(1,7)$ \\
			\hline 
			6 & 2 & 43 & 56 & 13 & $\mathrm{AGL}(1,8)$ \\
			\hline 
		\end{tabular}
		\caption{Smallest Cayley digraphs of given degree $d$ and geodecity $k$}
		\label{tab:caydig}
	\end{table}
	
	The general graph search results in Table~\ref{tab:gendig} show a similar pattern, although because the search space is very much larger than in the Cayley case the range of values for which we are able to determine the order of cages is quite restricted.
	
	\begin{table}[h]\centering
		\begin{tabular}{|lll|lll|}
			\hline
			$d$ & $k$ & $M$ & $n$ & $\epsilon$ & Comment \\
			\hline
			2 & 2 & 7 & 9 & 2 & Figure~\ref{fig:d2k2n7}\\
			& 3 & 15 & 20 & 5 & Figure~\ref{fig:d2k3n20} \\
			& 4 & 31 & 54* & 23* & No graphs of order less than 34 \\
			\hline
			3 & 2 & 13 & 16 & 3 & Figure~\ref{fig:d3k2n16}\\
			\hline
		\end{tabular}
		\caption{Smallest digraphs of given degree $d$ and geodecity $k$ (* = smallest known)}
		\label{tab:gendig}
	\end{table}
	
	\subsection{Mixed graphs}
	The Cayley graphs in Table~\ref{tab:caymix} were again found by searching groups of increasing order until the first Cayley graph with the required geodecity was found. Thus the entries in this table are all minimal. The search for general graphs is again much more difficult. We have been able to find bounds for the orders of some cages, but the search space is so large that only a few provably minimal entries are known.
	
	\begin{table}[h]\centering
		\begin{tabular}{|lll|ll|lll|}
			\hline
			$d$ & $r$ & $z$ & $k$ & $M$ & $n$ & $\epsilon$ & Group \\
			\hline
			2 & 1 & 1 & 2 & 6 & 6 & 0 & $S_3$ \\
			&  &  & 3 & 11 & 20 & 9 & $\mathrm{AGL}(1,5)$ \\
			&  &  & 4 & 19 & 32 & 13 & $(\Z_8\rtimes\Z_2)\rtimes\Z_2$ \\
			&  &  & 5 & 32 & 54 & 22 & $(\Z_9\rtimes\Z_3)\rtimes\Z_2$ \\
			\hline 
			3 & 2 & 1 & 2 & 11 & 12 & 1 & $D_{12}$ \\
			& & & 3 & 28 & 48 & 20 & $\Z_2\times S_4$ \\
			& 1 & 2 & 2 & 12 & 12 & 0 & $A_4$ \\
			& & & 3 & 34 & 64 & 30 & $((\Z_{8} \rtimes \Z_{2}) \rtimes \Z_{2}) \rtimes \Z_{2}$ \\
			\hline
			4 & 3 & 1 & 2 & 18 & 18 & 0 & $\Z_3\times S_3$ \\
			& 2 & 2 & 2 & 19 & 24 & 5 & $\mathrm{SL}(2,3)$ \\
			& 1 & 3 & 2 & 20 & 20 & 0 & $\mathrm{AGL}(1,5)$ \\
			\hline
			5 & 4 & 1 & 2 & 27 & 30 & 3 & $\Z_5\times S_3$ \\
			& 3 & 2 & 2 & 28 & 42 & 14 & $\Z_7\times S_3$ \\
			& 2 & 3 & 2 & 29 & 39 & 10 & $\Z_{13}\times\Z_3$ \\
			& 1 & 4 & 2 & 30 & 42 & 12 & $\mathrm{AGL}(1,7)$ \\
			\hline
			6 & 5 & 1 & 2 & 38 & 48 & 10 & $D_{48}$ \\
			& 4 & 2 & 2 & 39 & 48 & 9 & $D_8\times S_3$ \\
			& 3 & 3 & 2 & 40 & 52 & 12 & $\Z_{13}\rtimes\Z_4$ \\
			& 2 & 4 & 2 & 41 & 54 & 13 & $(\Z_3\times\Z_3)\rtimes\Z_6$ \\
			& 1 & 5 & 2 & 42 & 42 & 0 & $\mathrm{AGL}(1,7)$ \\
			\hline
		\end{tabular}
		\caption{Smallest Cayley mixed graphs of given total degree $d$ and geodecity $k$}
		\label{tab:caymix}
	\end{table}
	
	\begin{table}[h]\centering
		\begin{tabular}{|lll|ll|lll|}
			\hline
			$d$ & $r$ & $z$ & $k$ & $M$ & $n$ & $\epsilon$ & Comment \\
			\hline
			2 & 1 & 1 & 2 & 6 & 6 & 0 & Kautz graph \\
			&  &  & 3 & 11 & 16 & 5 & Figure~\ref{fig:r1z1k3}\\
			&  &  & 4 & 19 & 30 & 11 & Figure~\ref{fig:r1z1k4}\\
			&  &  & 5 & 32 & 54* & 22* & No graphs of order less than 50 \\
			\hline 
			3 & 2 & 1 & 2 & 11 & 12 & 1 & Cayley graph of $D_{12}$ \\
			& 2 & 1 & 3 & 28 & 48* & 20* & No graphs of order less than 32 \\
			& 1 & 2 & 2 & 12 & 12 & 0 & Kautz graph \\
			\hline
			4 & 3 & 1 & 2 & 18 & 18 & 0 & Bos\'ak graph \\
			& 2 & 2 & 2 & 19 & 21 & 2 & Figure~\ref{fig:r2z2k2} \\
			& 1 & 3 & 2 & & 20 & 0 & Kautz graph \\
			\hline
		\end{tabular}
		\caption{Smallest mixed graphs of given total degree $d$ and geodecity $k$ (* = smallest known)}
		\label{tab:genmix}
	\end{table}
	
	\section{Figures of some known cages}\label{sec:figs}
	
	\begin{figure}[h]\centering
		\begin{tikzpicture}[middlearrow=stealth,x=0.2mm,y=-0.2mm,inner sep=0.2mm,scale=0.6,very thick,vertex/.style={circle,draw,minimum size=10,fill=lightgray}]
			\node at (380,140) [vertex] (v1) {};
			\node at (380,220) [vertex] (v2) {};
			\node at (620,520) [vertex] (v3) {};
			\node at (380,300) [vertex] (v4) {};
			\node at (460,440) [vertex] (v5) {};
			\node at (540,480) [vertex] (v6) {};
			\node at (140,520) [vertex] (v7) {};
			\node at (220,480) [vertex] (v8) {};
			\node at (300,440) [vertex] (v9) {};
			\path
			(v1) edge [middlearrow] (v2)
			(v1) edge [middlearrow] (v3)
			(v2) edge [middlearrow] (v4)
			(v2) edge [middlearrow] (v5)
			(v3) edge [middlearrow] (v6)
			(v3) edge [middlearrow] (v7)
			(v4) edge [middlearrow,bend right] (v1)
			(v4) edge [middlearrow] (v6)
			(v5) edge [middlearrow,bend right] (v3)
			(v5) edge [middlearrow] (v8)
			(v6) edge [middlearrow] (v5)
			(v6) edge [middlearrow] (v9)
			(v7) edge [middlearrow] (v1)
			(v7) edge [middlearrow] (v8)
			(v8) edge [middlearrow] (v4)
			(v8) edge [middlearrow] (v9)
			(v9) edge [middlearrow] (v2)
			(v9) edge [middlearrow,bend right] (v7)
			;
		\end{tikzpicture}
		~~
		\begin{tikzpicture}[middlearrow=stealth,x=0.2mm,y=-0.2mm,inner sep=0.2mm,scale=0.6,very thick,vertex/.style={circle,draw,minimum size=10,fill=lightgray}]
			\node at (380,300) [vertex] (v1) {};
			\node at (380,140) [vertex] (v2) {};
			\node at (220,480) [vertex] (v3) {};
			\node at (380,220) [vertex] (v4) {};
			\node at (620,520) [vertex] (v5) {};
			\node at (300,440) [vertex] (v6) {};
			\node at (460,440) [vertex] (v7) {};
			\node at (540,480) [vertex] (v8) {};
			\node at (140,520) [vertex] (v9) {};
			\path
			(v1) edge [middlearrow,bend right] (v2)
			(v1) edge [middlearrow] (v3)
			(v2) edge [middlearrow] (v4)
			(v2) edge [middlearrow] (v5)
			(v3) edge [middlearrow] (v6)
			(v3) edge [middlearrow] (v7)
			(v4) edge [middlearrow] (v1)
			(v4) edge [middlearrow] (v6)
			(v5) edge [middlearrow] (v8)
			(v5) edge [middlearrow] (v9)
			(v6) edge [middlearrow] (v8)
			(v6) edge [middlearrow,bend right] (v9)
			(v7) edge [middlearrow] (v4)
			(v7) edge [middlearrow,bend right] (v5)
			(v8) edge [middlearrow] (v1)
			(v8) edge [middlearrow] (v7)
			(v9) edge [middlearrow] (v2)
			(v9) edge [middlearrow] (v3)
			;
		\end{tikzpicture}
		\caption{Two isomorphism classes of $(2,2,+2)$ digraphs}
		\label{fig:d2k2n7}
	\end{figure}
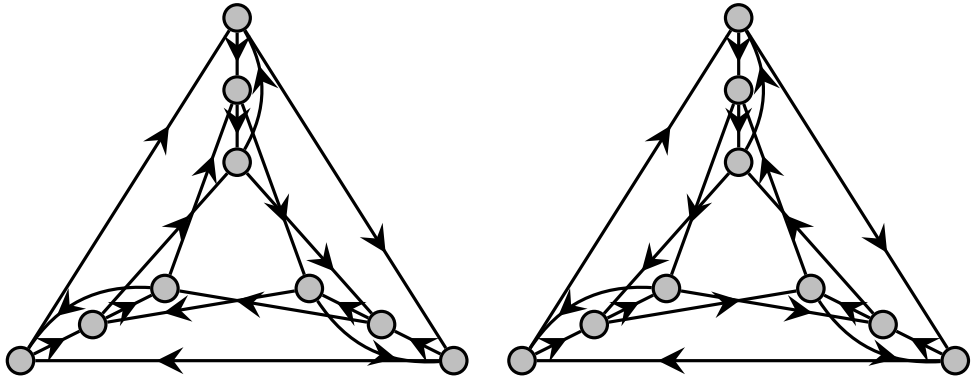
	
	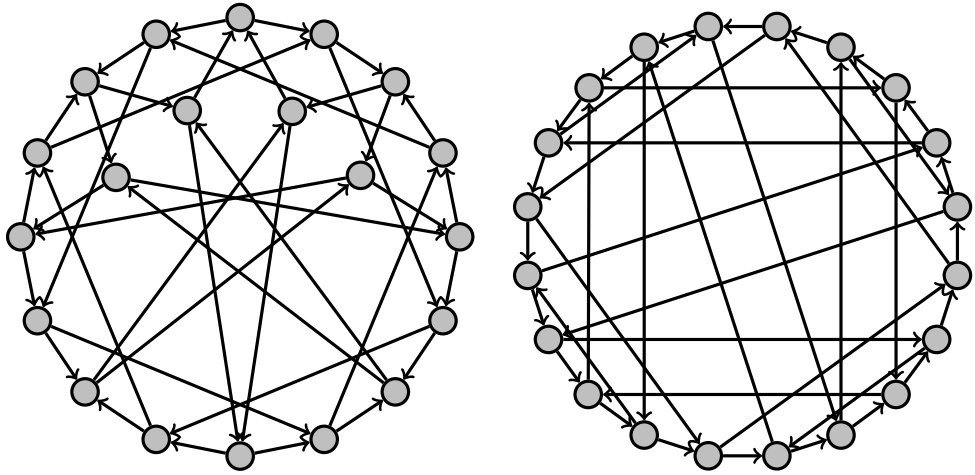
\begin{figure}[h]
		\centering
		\begin{tikzpicture}[x=0.4mm,y=-0.4mm,inner sep=0.2mm,scale=0.25,very thick,vertex/.style={circle,draw,minimum size=10,fill=lightgray}]
			\node at (181.4,173.2) [vertex] (v1) {};
			\node at (222.6,300.4) [vertex] (v2) {};
			\node at (317.2,211.5) [vertex] (v3) {};
			\node at (95.9,379.5) [vertex] (v4) {};
			\node at (679.4,379.5) [vertex] (v5) {};
			\node at (387.6,87.8) [vertex] (v6) {};
			\node at (387.6,671.2) [vertex] (v7) {};
			\node at (118.1,267.9) [vertex] (v8) {};
			\node at (118.1,491.1) [vertex] (v9) {};
			\node at (657.2,267.9) [vertex] (v10) {};
			\node at (657.2,491.1) [vertex] (v11) {};
			\node at (499.3,110) [vertex] (v12) {};
			\node at (276,110) [vertex] (v13) {};
			\node at (499.3,649) [vertex] (v14) {};
			\node at (276,649) [vertex] (v15) {};
			\node at (181.4,585.8) [vertex] (v16) {};
			\node at (593.9,173.2) [vertex] (v17) {};
			\node at (593.9,585.8) [vertex] (v18) {};
			\node at (547.6,297.7) [vertex] (v19) {};
			\node at (456.8,212.6) [vertex] (v20) {};
			\path
			(v1) edge[->] (v2)
			(v1) edge[->] (v3)
			(v8) edge[->] (v1)
			(v13) edge[->] (v1)
			(v2) edge[->] (v4)
			(v2) edge[->] (v5)
			(v18) edge[->] (v2)
			(v3) edge[->] (v6)
			(v3) edge[->] (v7)
			(v18) edge[->] (v3)
			(v4) edge[->] (v8)
			(v4) edge[->] (v9)
			(v19) edge[->] (v4)
			(v5) edge[->] (v10)
			(v5) edge[->] (v11)
			(v19) edge[->] (v5)
			(v6) edge[->] (v12)
			(v6) edge[->] (v13)
			(v20) edge[->] (v6)
			(v7) edge[->] (v14)
			(v7) edge[->] (v15)
			(v20) edge[->] (v7)
			(v8) edge[->] (v12)
			(v15) edge[->] (v8)
			(v13) edge[->] (v9)
			(v9) edge[->] (v14)
			(v9) edge[->] (v16)
			(v10) edge[->] (v13)
			(v14) edge[->] (v10)
			(v10) edge[->] (v17)
			(v12) edge[->] (v11)
			(v11) edge[->] (v15)
			(v11) edge[->] (v18)
			(v12) edge[->] (v17)
			(v14) edge[->] (v18)
			(v15) edge[->] (v16)
			(v16) edge[->] (v19)
			(v16) edge[->] (v20)
			(v17) edge[->] (v19)
			(v17) edge[->] (v20)
			;
		\end{tikzpicture}
		\quad
		\begin{tikzpicture}[x=0.4mm,y=-0.4mm,inner sep=0.2mm,scale=0.3,very thick,vertex/.style={circle,draw,minimum size=10,fill=lightgray}]
			\node at (532,328) [vertex] (v1) {};
			\node at (403,75) [vertex] (v2) {};
			\node at (56,252) [vertex] (v3) {};
			\node at (256,52) [vertex] (v4) {};
			\node at (79,399) [vertex] (v5) {};
			\node at (509,181) [vertex] (v6) {};
			\node at (332,528) [vertex] (v7) {};
			\node at (185,505) [vertex] (v8) {};
			\node at (464,460) [vertex] (v9) {};
			\node at (124,120) [vertex] (v10) {};
			\node at (332,52) [vertex] (v11) {};
			\node at (532,252) [vertex] (v12) {};
			\node at (256,528) [vertex] (v13) {};
			\node at (509,399) [vertex] (v14) {};
			\node at (79,181) [vertex] (v15) {};
			\node at (56,328) [vertex] (v16) {};
			\node at (464,120) [vertex] (v17) {};
			\node at (403,505) [vertex] (v18) {};
			\node at (185,75) [vertex] (v19) {};
			\node at (123,460) [vertex] (v20) {};
			\path
			(v1) edge[->] (v11)
			(v1) edge[->] (v12)
			(v1) edge[<-] (v13)
			(v1) edge[<-] (v14)
			(v2) edge[->] (v11)
			(v2) edge[->] (v12)
			(v2) edge[<-] (v17)
			(v2) edge[<-] (v18)
			(v3) edge[<-] (v11)
			(v3) edge[->] (v13)
			(v3) edge[<-] (v15)
			(v3) edge[->] (v16)
			(v4) edge[<-] (v11)
			(v4) edge[<-] (v15)
			(v4) edge[->] (v18)
			(v4) edge[->] (v19)
			(v5) edge[<-] (v12)
			(v5) edge[->] (v14)
			(v5) edge[<-] (v16)
			(v5) edge[->] (v20)
			(v6) edge[<-] (v12)
			(v6) edge[->] (v15)
			(v6) edge[<-] (v16)
			(v6) edge[->] (v17)
			(v7) edge[<-] (v13)
			(v7) edge[<-] (v14)
			(v7) edge[->] (v18)
			(v7) edge[->] (v19)
			(v8) edge[->] (v13)
			(v8) edge[->] (v16)
			(v8) edge[<-] (v19)
			(v8) edge[<-] (v20)
			(v9) edge[->] (v14)
			(v9) edge[<-] (v17)
			(v9) edge[<-] (v18)
			(v9) edge[->] (v20)
			(v10) edge[->] (v15)
			(v10) edge[->] (v17)
			(v10) edge[<-] (v19)
			(v10) edge[<-] (v20)
			;
		\end{tikzpicture}
		\caption{Two digraphs with $d=2,k=3,\epsilon=5$}
		\label{fig:d2k3n20}
	\end{figure}
	
	\begin{figure}[h]\centering
		\begin{tikzpicture}[middlearrow=stealth,x=0.2mm,y=-0.2mm,inner sep=0.2mm,scale=1,thick,vertex/.style={circle,draw,minimum size=10,fill=lightgray}]
			\node at (228,172) [vertex] (v1) {};
			\node at (183,90) [vertex] (v2) {};
			\node at (421,502) [vertex] (v3) {};
			\node at (372,172) [vertex] (v4) {};
			\node at (228,420) [vertex] (v5) {};
			\node at (157,296) [vertex] (v6) {};
			\node at (64,296) [vertex] (v7) {};
			\node at (540,296) [vertex] (v8) {};
			\node at (299,260) [vertex] (v9) {};
			\node at (299,339) [vertex] (v10) {};
			\node at (443,296) [vertex] (v11) {};
			\node at (421,90) [vertex] (v12) {};
			\node at (183,502) [vertex] (v13) {};
			\node at (259,300) [vertex] (v14) {};
			\node at (339,300) [vertex] (v15) {};
			\node at (372,420) [vertex] (v16) {};
			\path
			(v1) edge[middlearrow] (v2)
			(v1) edge[middlearrow,bend left] (v3)
			(v1) edge[middlearrow] (v4)
			(v6) edge[middlearrow] (v1)
			(v9) edge[middlearrow,bend left] (v1)
			(v10) edge[middlearrow,bend left] (v1)
			(v2) edge[middlearrow] (v7)
			(v2) edge[middlearrow,bend left] (v9)
			(v12) edge[middlearrow] (v2)
			(v2) edge[middlearrow,bend left] (v14)
			(v16) edge[middlearrow,bend left] (v2)
			(v3) edge[middlearrow] (v8)
			(v3) edge[middlearrow,bend left] (v10)
			(v13) edge[middlearrow] (v3)
			(v3) edge[middlearrow,bend left] (v15)
			(v16) edge[middlearrow] (v3)
			(v4) edge[middlearrow] (v11)
			(v4) edge[middlearrow] (v12)
			(v4) edge[middlearrow,bend left] (v13)
			(v14) edge[middlearrow,bend left] (v4)
			(v15) edge[middlearrow,bend left] (v4)
			(v5) edge[middlearrow] (v6)
			(v9) edge[middlearrow,bend left] (v5)
			(v10) edge[middlearrow,bend left] (v5)
			(v5) edge[middlearrow,bend left] (v12)
			(v5) edge[middlearrow] (v13)
			(v16) edge[middlearrow] (v5)
			(v6) edge[middlearrow] (v7)
			(v6) edge[middlearrow,bend left] (v8)
			(v14) edge[middlearrow,bend left] (v6)
			(v15) edge[middlearrow,bend left] (v6)
			(v7) edge[middlearrow,bend left] (v10)
			(v11) edge[middlearrow,bend left] (v7)
			(v7) edge[middlearrow] (v13)
			(v7) edge[middlearrow,bend left] (v15)
			(v8) edge[middlearrow,bend left] (v9)
			(v11) edge[middlearrow] (v8)
			(v8) edge[middlearrow] (v12)
			(v8) edge[middlearrow,bend left] (v14)
			(v9) edge[middlearrow,bend left] (v11)
			(v13) edge[middlearrow,bend left] (v9)
			(v10) edge[middlearrow,bend left] (v11)
			(v12) edge[middlearrow,bend left] (v10)
			(v11) edge[middlearrow] (v16)
			(v12) edge[middlearrow,bend left] (v15)
			(v13) edge[middlearrow,bend left] (v14)
			(v14) edge[middlearrow,bend left] (v16)
			(v15) edge[middlearrow,bend left] (v16)
			;
		\end{tikzpicture}
		\caption{The unique extremal digraph $d=3,k=2,n=16$}
		\label{fig:d3k2n16}
	\end{figure}
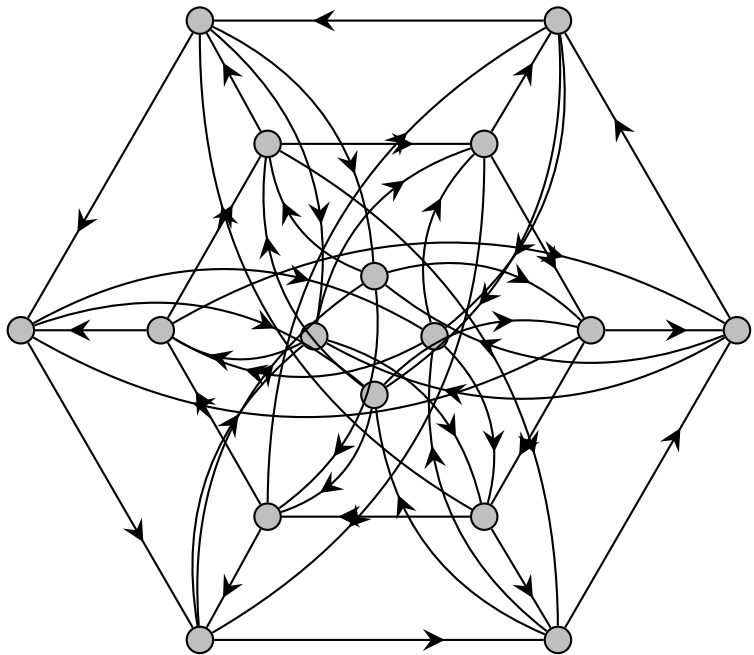

	\begin{figure}[h]
		\centering
		\begin{tikzpicture}[middlearrow=stealth,x=0.2mm,y=-0.2mm,inner sep=0.2mm,scale=0.7,thick,vertex/.style={circle,draw,minimum size=10,fill=lightgray}]
			\node at (299,65) [vertex] (v1) {};
			\node at (465,133) [vertex] (v2) {};
			\node at (64,300) [vertex] (v3) {};
			\node at (133,133) [vertex] (v4) {};
			\node at (534,300) [vertex] (v5) {};
			\node at (299,472) [vertex] (v6) {};
			\node at (177,422) [vertex] (v7) {};
			\node at (465,466) [vertex] (v8) {};
			\node at (299,127) [vertex] (v9) {};
			\node at (421,178) [vertex] (v10) {};
			\node at (127,300) [vertex] (v11) {};
			\node at (177,178) [vertex] (v12) {};
			\node at (472,300) [vertex] (v13) {};
			\node at (299,535) [vertex] (v14) {};
			\node at (133,466) [vertex] (v15) {};
			\node at (421,422) [vertex] (v16) {};
			\path
			(v1) edge (v9)
			(v2) edge (v10)
			(v3) edge (v11)
			(v4) edge (v12)
			(v5) edge (v13)
			(v6) edge (v14)
			(v7) edge (v15)
			(v8) edge (v16)
			(v1) edge[middlearrow] (v2)
			(v9) edge[middlearrow] (v11)
			(v2) edge[middlearrow] (v5)
			(v10) edge[middlearrow] (v12)
			(v3) edge[middlearrow] (v4)
			(v11) edge[middlearrow] (v6)
			(v4) edge[middlearrow] (v1)
			(v12) edge[middlearrow] (v7)
			(v5) edge[middlearrow] (v8)
			(v13) edge[middlearrow] (v9)
			(v6) edge[middlearrow] (v13)
			(v14) edge[middlearrow] (v15)
			(v7) edge[middlearrow] (v16)
			(v15) edge[middlearrow] (v3)
			(v8) edge[middlearrow] (v14)
			(v16) edge[middlearrow] (v10)
			;
		\end{tikzpicture}
		\quad
		\begin{tikzpicture}[middlearrow=stealth,x=0.2mm,y=-0.2mm,inner sep=0.2mm,scale=0.7,thick,vertex/.style={circle,draw,minimum size=10,fill=lightgray}]
			\node at (300,66) [vertex] (v1) {};
			\node at (417,98) [vertex] (v2) {};
			\node at (67,300) [vertex] (v3) {};
			\node at (300,385) [vertex] (v4) {};
			\node at (502,183) [vertex] (v5) {};
			\node at (183,502) [vertex] (v6) {};
			\node at (98,417) [vertex] (v7) {};
			\node at (534,300) [vertex] (v8) {};
			\node at (300,215) [vertex] (v9) {};
			\node at (417,502) [vertex] (v10) {};
			\node at (215,300) [vertex] (v11) {};
			\node at (300,533) [vertex] (v12) {};
			\node at (98,183) [vertex] (v13) {};
			\node at (183,98) [vertex] (v14) {};
			\node at (502,417) [vertex] (v15) {};
			\node at (385,300) [vertex] (v16) {};
			\path
			(v1) edge (v9)
			(v2) edge (v10)
			(v3) edge (v11)
			(v4) edge (v12)
			(v5) edge (v13)
			(v6) edge (v14)
			(v7) edge (v15)
			(v8) edge (v16)
			(v1) edge[middlearrow] (v2)
			(v9) edge[middlearrow] (v11)
			(v2) edge[middlearrow] (v5)
			(v10) edge[middlearrow] (v12)
			(v3) edge[middlearrow] (v13)
			(v11) edge[middlearrow] (v4)
			(v4) edge[middlearrow] (v16)
			(v12) edge[middlearrow] (v6)
			(v5) edge[middlearrow] (v8)
			(v13) edge[middlearrow] (v14)
			(v6) edge[middlearrow] (v7)
			(v14) edge[middlearrow] (v1)
			(v7) edge[middlearrow] (v3)
			(v15) edge[middlearrow] (v10)
			(v8) edge[middlearrow] (v15)
			(v16) edge[middlearrow] (v9)
			;
		\end{tikzpicture}
		\caption{Two mixed graphs with $r=1,z=1,k=3,\epsilon=5$}
		\label{fig:r1z1k3}
	\end{figure}
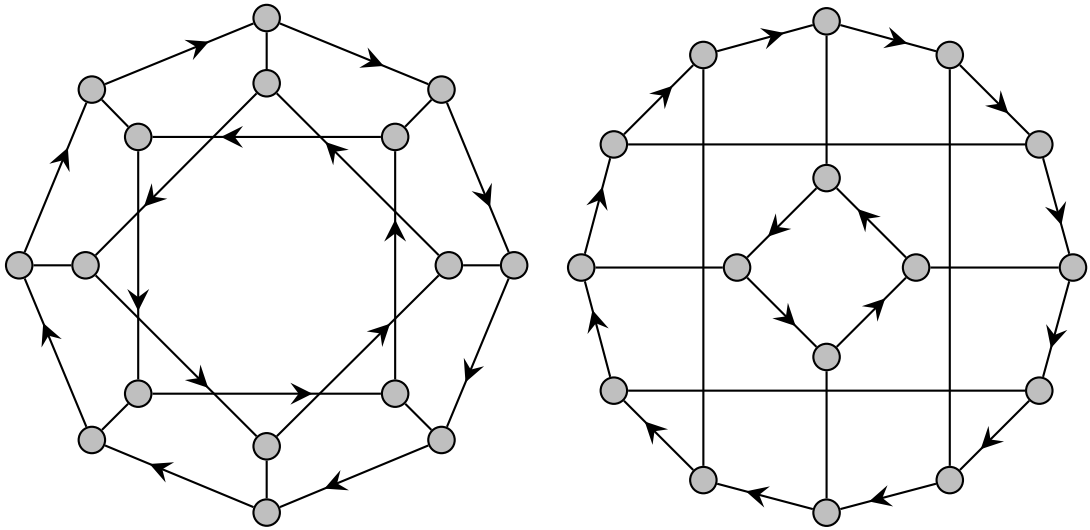
	
	\begin{figure}[h]
		\centering
		\begin{tikzpicture}[middlearrow=stealth,x=0.2mm,y=-0.2mm,inner sep=0.2mm,scale=0.7,thick,vertex/.style={circle,draw,minimum size=10,fill=lightgray}]
			\node at (276,338) [vertex] (v1) {};
			\node at (331,338) [vertex] (v2) {};
			\node at (347,286) [vertex] (v3) {};
			\node at (303,254) [vertex] (v4) {};
			\node at (260,286) [vertex] (v5) {};
			\node at (250,373) [vertex] (v6) {};
			\node at (358,373) [vertex] (v7) {};
			\node at (392,270) [vertex] (v8) {};
			\node at (304,206) [vertex] (v9) {};
			\node at (216,270) [vertex] (v10) {};
			\node at (432,260) [vertex] (v11) {};
			\node at (300,164) [vertex] (v12) {};
			\node at (169,260) [vertex] (v13) {};
			\node at (219,415) [vertex] (v14) {};
			\node at (382,414) [vertex] (v15) {};
			\node at (474,244) [vertex] (v16) {};
			\node at (298,118) [vertex] (v17) {};
			\node at (123,245) [vertex] (v18) {};
			\node at (191,451) [vertex] (v19) {};
			\node at (407,451) [vertex] (v20) {};
			\node at (298,69) [vertex] (v21) {};
			\node at (79,229) [vertex] (v22) {};
			\node at (163,487) [vertex] (v23) {};
			\node at (434,486) [vertex] (v24) {};
			\node at (517,228) [vertex] (v25) {};
			\node at (33,212) [vertex] (v26) {};
			\node at (135,521) [vertex] (v27) {};
			\node at (459,521) [vertex] (v28) {};
			\node at (559,211) [vertex] (v29) {};
			\node at (295,20) [vertex] (v30) {};
			\path
			(v1) edge[middlearrow] (v2)
			(v5) edge[middlearrow] (v1)
			(v1) edge (v6)
			(v2) edge[middlearrow] (v3)
			(v2) edge (v11)
			(v3) edge[middlearrow] (v4)
			(v3) edge (v16)
			(v4) edge[middlearrow] (v5)
			(v4) edge (v21)
			(v5) edge (v26)
			(v6) edge[middlearrow] (v7)
			(v10) edge[middlearrow] (v6)
			(v7) edge[middlearrow] (v8)
			(v7) edge (v18)
			(v8) edge[middlearrow] (v9)
			(v8) edge (v14)
			(v9) edge[middlearrow] (v10)
			(v9) edge (v28)
			(v10) edge (v24)
			(v11) edge[middlearrow] (v12)
			(v15) edge[middlearrow] (v11)
			(v12) edge[middlearrow] (v13)
			(v12) edge (v23)
			(v13) edge[middlearrow] (v14)
			(v13) edge (v29)
			(v14) edge[middlearrow] (v15)
			(v15) edge (v17)
			(v16) edge[middlearrow] (v17)
			(v20) edge[middlearrow] (v16)
			(v17) edge[middlearrow] (v18)
			(v18) edge[middlearrow] (v19)
			(v19) edge[middlearrow] (v20)
			(v19) edge (v30)
			(v20) edge (v22)
			(v21) edge[middlearrow] (v22)
			(v25) edge[middlearrow] (v21)
			(v22) edge[middlearrow] (v23)
			(v23) edge[middlearrow] (v24)
			(v24) edge[middlearrow] (v25)
			(v25) edge (v27)
			(v26) edge[middlearrow] (v27)
			(v30) edge[middlearrow] (v26)
			(v27) edge[middlearrow] (v28)
			(v28) edge[middlearrow] (v29)
			(v29) edge[middlearrow] (v30)
			;
		\end{tikzpicture}
		\caption{A mixed graph with $r=1,z=1,k=4,\epsilon=11$}
		\label{fig:r1z1k4}
	\end{figure}
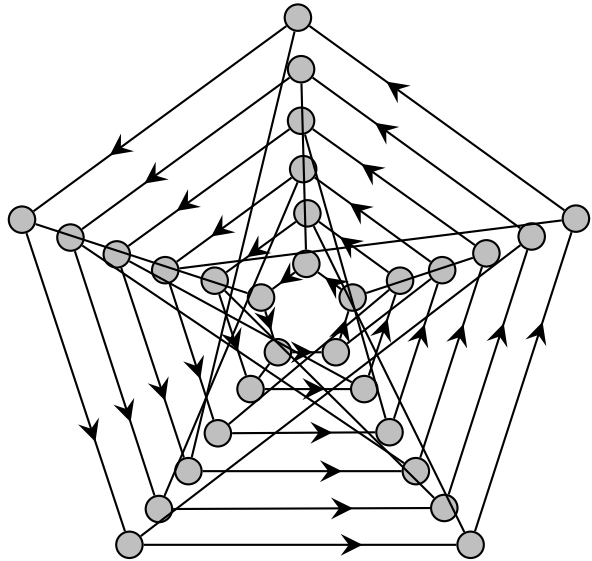
	
	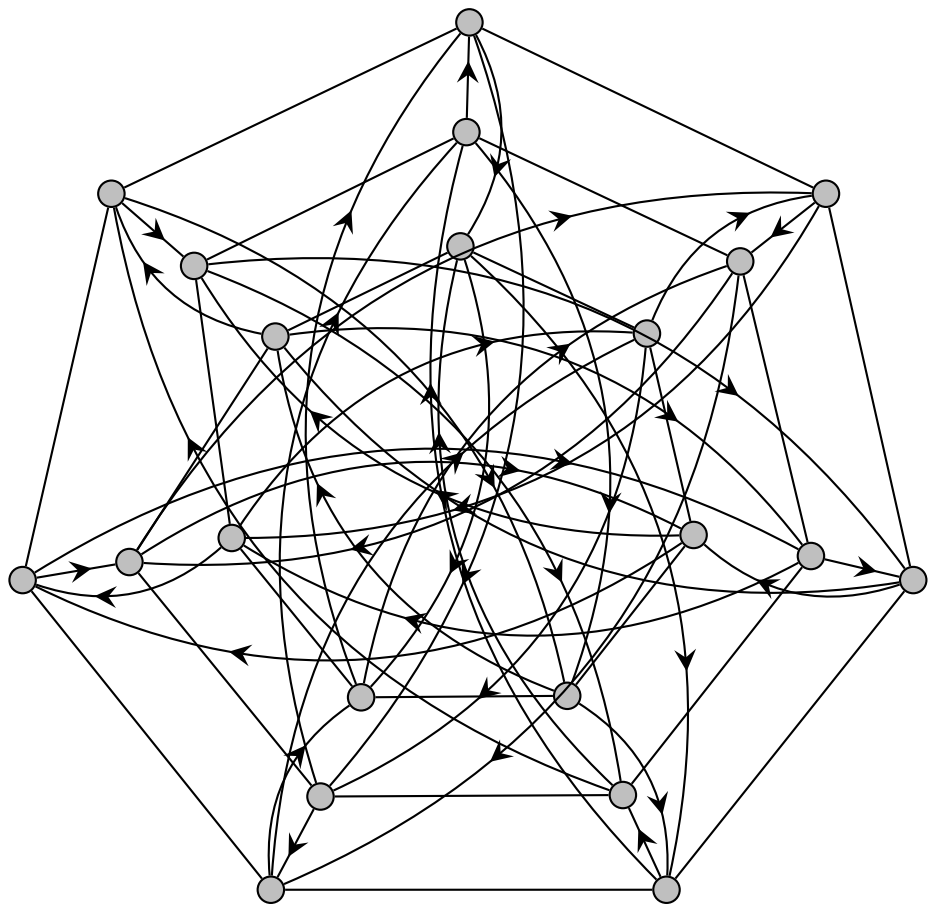
\begin{figure}\centering
		\begin{tikzpicture}[middlearrow=stealth,x=0.2mm,y=-0.2mm,inner sep=0.2mm,scale=1,thick,vertex/.style={circle,draw,minimum size=10,fill=lightgray}]
			\node at (588,446) [vertex] (v1) {};
			\node at (541,250) [vertex] (v2) {};
			\node at (123,205) [vertex] (v3) {};
			\node at (463,605) [vertex] (v4) {};
			\node at (359,164) [vertex] (v5) {};
			\node at (361,91) [vertex] (v6) {};
			\node at (135,450) [vertex] (v7) {};
			\node at (229,668) [vertex] (v8) {};
			\node at (656,462) [vertex] (v9) {};
			\node at (203,434) [vertex] (v10) {};
			\node at (492,668) [vertex] (v11) {};
			\node at (289,540) [vertex] (v12) {};
			\node at (232,300) [vertex] (v13) {};
			\node at (355,240) [vertex] (v14) {};
			\node at (598,205) [vertex] (v15) {};
			\node at (262,606) [vertex] (v16) {};
			\node at (426,539) [vertex] (v17) {};
			\node at (510,432) [vertex] (v18) {};
			\node at (479,298) [vertex] (v19) {};
			\node at (178,253) [vertex] (v20) {};
			\node at (64,462) [vertex] (v21) {};
			\path
			(v1) edge (v2)
			(v1) edge (v4)
			(v1) edge[middlearrow] (v9)
			(v1) edge[middlearrow,bend left] (v10)
			(v13) edge[middlearrow,bend left] (v1)
			(v21) edge[middlearrow,bend left] (v1)
			(v2) edge (v5)
			(v2) edge[middlearrow,bend left] (v7)
			(v2) edge[middlearrow,bend left] (v8)
			(v12) edge[middlearrow,bend left] (v2)
			(v15) edge[middlearrow] (v2)
			(v4) edge[middlearrow,bend left] (v3)
			(v3) edge (v6)
			(v13) edge[middlearrow,bend left] (v3)
			(v3) edge[middlearrow,bend left] (v17)
			(v3) edge[middlearrow] (v20)
			(v3) edge (v21)
			(v11) edge[middlearrow] (v4)
			(v4) edge[middlearrow,bend left] (v14)
			(v4) edge (v16)
			(v20) edge[middlearrow,bend left] (v4)
			(v5) edge[middlearrow] (v6)
			(v11) edge[middlearrow,bend left] (v5)
			(v16) edge[middlearrow,bend left] (v5)
			(v5) edge[middlearrow,bend left] (v17)
			(v5) edge (v20)
			(v12) edge[middlearrow,bend left] (v6)
			(v6) edge[middlearrow,bend left] (v14)
			(v6) edge (v15)
			(v6) edge[middlearrow,bend left] (v16)
			(v7) edge (v13)
			(v7) edge[middlearrow,bend left] (v15)
			(v7) edge (v16)
			(v7) edge[middlearrow,bend left] (v18)
			(v21) edge[middlearrow] (v7)
			(v8) edge (v11)
			(v8) edge[middlearrow,bend left] (v12)
			(v16) edge[middlearrow] (v8)
			(v8) edge[middlearrow,bend left] (v19)
			(v8) edge (v21)
			(v9) edge (v11)
			(v9) edge[middlearrow,bend left] (v13)
			(v9) edge (v15)
			(v9) edge[middlearrow,bend left] (v18)
			(v20) edge[middlearrow,bend left] (v9)
			(v10) edge (v12)
			(v15) edge[middlearrow,bend left] (v10)
			(v10) edge[middlearrow,bend left] (v19)
			(v10) edge (v20)
			(v10) edge[middlearrow,bend left] (v21)
			(v14) edge[middlearrow,bend left] (v11)
			(v17) edge[middlearrow,bend left] (v11)
			(v14) edge[middlearrow,bend left] (v12)
			(v12) edge (v17)
			(v13) edge (v14)
			(v17) edge[middlearrow,bend left] (v13)
			(v14) edge (v19)
			(v19) edge[middlearrow,bend left] (v15)
			(v19) edge[middlearrow,bend left] (v16)
			(v17) edge (v18)
			(v18) edge (v19)
			(v18) edge[middlearrow,bend left] (v20)
			(v18) edge[middlearrow,bend left] (v21)
			;
		\end{tikzpicture}
		\caption{The extremal mixed graph $r=2,z=2,k=2,n=21$}
		\label{fig:r2z2k2}
	\end{figure}

	\section*{Acknowledgements}
	The first author acknowledges funding from an LMS Early Career Fellowship and thanks the Open University for an extension of funding in 2020.

	%----------------------------------------------
	
\end{document}